\documentclass[11pt]{amsart}
\usepackage{amsmath,amsthm, amscd, amssymb, amsfonts, mathtools}
\usepackage[all]{xy}
\usepackage{enumitem}
\usepackage{hyperref}
\usepackage{t1enc}
\usepackage[latin1]{inputenc}
\usepackage{fontsmpl}
\usepackage{srcltx}
\usepackage{placeins}
\usepackage{float}
\usepackage{color}

\allowdisplaybreaks

\numberwithin{equation}{section}\theoremstyle{plain}

\newtheorem{teo}{Theorem}[subsection]

\newtheorem{prop}[teo]{Proposition}
\newtheorem{lem}[teo]{Lemma}
\newtheorem{lema}[teo]{Lemma}
\newtheorem{cor}[teo]{Corollary}
\theoremstyle{definition}

\newtheorem{rem}[teo]{Remark}

\newtheorem{ex}[teo]{Example}
\newtheorem{exas}[teo]{Examples}

\newcommand\I{\mathbb I}

\newcommand{\lda}{\lambda}
\newcommand{\ap}{\cdot}

\def \N {\mathbb{N}}
\def \Z {\mathbb{Z}}
\def \k {\Bbbk}
\def \o {\otimes}

\def \pf {\begin{proof}}
\def \epf {\end{proof}}

\def\lg{\langle}
\def\rg{\rangle}

\begin{document}


 \title[Classifying partial (co)actions]{Classifying partial (co)actions of Taft and Nichols Hopf algebras on their base fields}
\author[Fonseca, Martini and Silva]{Graziela Fonseca, Grasiela Martini and Leonardo Silva}

\address[Fonseca]{Instituto Federal Sul-rio-grandense, Brazil}
\email{grazielalangone@gmail.com}

\address[Martini]{Universidade Federal do Rio Grande, Brazil}
\email{grasiela.martini@furg.br}

\address[Silva]{IME, Universidade Federal do Rio Grande do Sul, Rio Grande do Sul, Brazil.}
\email{leonardoufpel@gmail.com}

\thanks{
	\hspace*{0.3cm} 
	\hspace*{0.3cm} \emph{The third author was partially supported by CNPq, Brazil.}}

\begin{abstract}
In this paper we determine all partial actions and partial coactions of Taft and Nichols Hopf algebras on their base fields.
Furthermore, we prove that all such partial (co)actions are symmetric.
\end{abstract}

\maketitle
\section{Introduction}
The notion of a partial group action first appeared within operator algebras theory.
In fact, in \cite{exel1994circle}, Exel describe the structure of suitable $C^*$-algebras on which the unitary circle $\mathbb{S}_1$ acts partially by automorphisms.
In \cite{Dokuchaev_exel} the notion of a group acting partially on an algebra was formulated in a purely algebraic context, motivating investigations in multiple directions.
In particular, a Galois theory for commutative ring extensions in a setting of partial group actions was developed in \cite{paques_ferrero_dokuchaev}.
This work inspired Caenepeel and Janssen to extend Galois theory to the context of Hopf algebras in \cite{caenepeel2008partial}, where the notions of partial actions and partial coactions of Hopf algebras on algebras were introduced.
After, partial (co)actions became an area of research in itself, with extensions and applications to other settings, such as groupoids \cite{paques_bagio}, categories \cite{Alvares_Alves_Batista} and among many others, as can be noticed in \cite{Dokuchaev_survey} and the references therein.

\smallbreak

Some of the current research interests are partial (co)representations of Hopf algebras \cite{corepresentations, Partial_representations, dilations}.
A difficulty in dealing with partial (co)actions occurs because a partial (co)action, unlike a global one, is not a morphism of (co)algebras since the linear map that defines a partial (co)action is not (co)multiplicative.
The concepts of partial (co)representations provide a way to obtain some control in the (co)multiplicative behavior of such a map.
This approach gives a categorical point of view for partial (co)actions. 
In particular, symmetric partial (co)actions of a Hopf algebra $H$ on its base field provide one-dimensional (co)representations of $H$.

Partial (co)actions of a Hopf algebra on its base field have been characterized and calculated by many authors, even for more general settings, see for instance \cite{Alvares_Alves_Batista, glob_twisted_patial_hopf, guris}.
However, few Hopf algebras have the partial (co)actions on their base field classified, namely group algebras, dual of finite group algebras and Sweedler's four-dimensional Hopf algebra.

\smallbreak

There are two important classes of finite dimensional Hopf algebras: the semisimple algebras and the pointed ones. With suitable conditions on the base field, the group algebras and the dual of finite group algebras are semisimple.
On the other hand, Sweedler's four-dimensional Hopf algebra is the smaller (non-semisimple) pointed Hopf algebra.

In \cite{Taft} Taft introduced a two parameters family of finite-dimensional Hopf algebras whose antipode map has finite order but as large as desired.
This family generalize the Sweedler's Hopf algebra and it happens to be pointed Hopf algebras.
Precisely, for positive integers $n$ and $\ell$, with $n \geq 2$, the Hopf algebra $H(\ell,n)$ constructed has dimension $n^{\ell+1}$ and its antipode map has order $2n$.
Moreover, the Hopf algebra $H(\ell,n)$  is not semisimple, and neither it nor its dual algebra are unimodular.
The Hopf algebra $H(1,n)$ is the so-called \emph{Taft algebra of order $n$} and has been extensively studied.
In particular, $H(1,2)$ is exactly the Sweedler's Hopf algebra.

Another family of pointed Hopf algebras is the \emph{Nichols Hopf algebras}\footnote{\emph{Nichols Hopf algebra} should not be confused with \emph{Nichols algebra}.}.
These Hopf algebras also are a particular case of the Hopf algebras constructed in the aforementioned article \cite{Taft}, namely the family $H(\ell, 2)$.
Such Hopf algebras were named in honor of Nichols 
after his work \cite{nichols}.
Nichols Hopf algebras have codimension 2.
Furthermore, the unique (up to isomorphism) pointed Hopf algebra of dimension $2^{\ell+1}$ with coradical $\k C_2$ is precisely $H(\ell, 2)$ \cite{Nichols_Caenepeel}.

\smallbreak 

The classification of global Hopf actions on algebras is a current research subject and widely studied, mainly when the Hopf algebra is semisimple or pointed.
In these cases, global Hopf actions on fields play an important role and they were particularly treated by Etingof and Walton in \cite{etingofewaltonSemisimple,etingofewaltonI,etingofewaltonII}, where they obtained a complete classification of them for semisimple Hopf algebras and several families of pointed Hopf algebras, for instance, the Taft and Nichols Hopf algebras.
Moreover, there exist currently a list of papers  that classify global actions of certain Hopf algebras on different types of algebras, such as Weyl algebras, algebras of differential operators of smooth affine varieties, path algebras of quivers and deformation quantizations of commutative domains \cite{CuadraEtingofWaltonI,CuadraEtingofWaltonII, etingofewaltonFinite, etingofewaltonFinite2,Acoes_taft_Chelsea}.

Furthermore, the global actions of Taft algebras on any finite dimensional algebra were recently classified \cite{Acoes_taft_Centrone}. 
However, calculate a partial action can be a more difficult task:
since a Taft algebra has two generator elements as algebra, to calculate a global action of a Taft algebra on any algebra is sufficient to deal only with these two generators;
on the other hand, to calculate a partial action of a Taft algebra one needs to deal with a whole basis, since the linear map that defines a partial action is not multiplicative, as already mentioned.

The aim of this paper is to classify partial actions and partial coactions on the base field for the two families of Hopf algebras mentioned earlier.
	In this way, we intend to provide a large class of examples in order to obtain a better understanding of the behavior of these partial (co)actions, which agree with the well-known results for two distinct types of Hopf algebras: the group algebras and the universal enveloping algebras of Lie algebras.
	These results extend those obtained for the Sweedler's Hopf algebra, overcoming some difficulties about computations in large dimensions.
	Moreover, this classification still provides an important class of examples of one-dimensional partial (co)representations, a question not yet investigated for any Hopf algebra, perhaps due to the lack of interesting examples.

\medbreak

The present work is organized as follows.
In Section \ref{sec:preliminaries} we recall the definitions and properties of partial (co)actions that are used throughout the paper.
Some results for the particular case of partial (co)actions of a bialgebra on its base field are presented.
In Section \ref{sec:Taft}, all partial (co)actions of a Taft algebra on its base field are classified.
Aiming this goal, some results concerning \emph{q-combinatorics} are developed.
Further, we prove that all such partial (co)actions are symmetric.
Finally, in Section \ref{sec:Nichols}, we classify all partial (co)actions of a Nichols Hopf algebra on its base field and we prove that all of them are symmetric.

\subsection*{Conventions}
All vector spaces and (co)algebras are considered over a field $\k$, and a (co)algebra means a (co)unital (co)algebra.
Besides, $\k^{\times} = \k \backslash \{0\}$ and unadorned $\otimes$ means $\otimes_{\Bbbk}$.

For a coalgebra $C$, $\varepsilon_C$ and $\Delta_C$ stands for the counit and the comultiplication maps of $C$, and will be written as  $\varepsilon$ and $\Delta$ if there is no ambiguity.
Moreover, we use Sweedler-Heynemann notation with summation sign suppressed for the comultiplication map, namely $\Delta(c)=c_1 \o c_2 \in C \o C$, for all $c \in C$.
We write $G(C)= \{g \in C \backslash \{0\} \ : \ \Delta(g) = g \o g \}$ for the set of the \emph{group-like elements} of the coalgebra $C$.
Given $g,h \in G(C)$, an element $x \in C$ is called a \emph{$(g,h)$-primitive element} if $\Delta(x)= x \o g + h \o x$.
If no emphasis on the elements $g, h \in G(C)$ is needed, a $(g,h)$-primitive element will be called simply of a \emph{skew-primitive element}.

Given a group $G$, we write $\k G$ for the group algebra with the canonical basis $\{g \ : \ g \in G \}$.
For $g \in G$, $\lg g \rg $ denotes the subgroup of $G$ generated by $g$. 
If $\mathcal{B} = \{v_1, v_2, \cdots , v_n\}$ is a basis for a finite-dimensional vector space $V$, then $ \{v_1^*, v_2^*, \cdots , v_n^*\}$ is the dual basis of $\mathcal{B}$.

We use $\Z$, $\N$ and $\N_0$ for the sets of integers, the positive integers and $\N \cup \{0\}$, respectively.
Let $ j, k \in \Z$ and $n \in \N$.
The symbols $\delta_{j,k}$ and $C_n$ stand for Kronecker's delta and the cyclic group of order $n$, respectively.
If $j \leq k$, then $\I_{j, k} =\{j, j+1, \cdots, k\}$ and $\I_n = \I_{1, n}$.

Next, we present some definitions and facts about \emph{q-combinatorics} that will be useful in Section \ref{sec:Taft}, where we deal with Taft algebras.

Let $n \in \N$ and $q \in \k^{\times}$.
The \emph{q-numbers} are defined recursively as $(0)_q = 0$ and  $(n)_q = 1+ q + \cdots + q^{n-1} = \sum_{\ell=0}^{n-1}q^\ell.$
Also, the  \emph{q-factorial} are defined recursively as 
$(0)_q ! = 1$ and $(n)_q ! = (n)_q (n-1)_q !.$

Thus, the \emph{q-binomial coefficients} are defined as follows.
First set ${0 \choose 0}_q = 1$.
Now, if $(n)_q ! = 0$, then ${n \choose 0}_q = 1 = {n \choose n}_q$ and ${n \choose k}_q = 0$, for all $k \in \I_{n-1}$.
Otherwise, if  $(n)_q ! \neq 0$,  then
${n \choose k}_q = \frac{(n)_q !}{(k)_q ! (n-k)_q !}$, for all $k \in \I_{0, n}$.

For any $m \in \N_0$, clearly ${m \choose \ell}_q = {m \choose {m-\ell}}_q$ for all $\ell \in \I_{0,m}$.
Moreover, if $\ell \in \Z$, then it is convenient to set ${m \choose \ell}_q = 0$  if $\ell>m$ or $\ell<0$.
With this convention, ${m \choose \ell}_q = {m \choose {m-\ell}}_q$ holds for all $\ell \in \Z$.

When $q=1$, all the notions above are the classical ones.
If $q \neq 1$ is an $n^{th}$ root of unity, then $(n)_q = 0$.
Consequently $(n)_q != 0$ and therefore ${n \choose k}_q = 0$, for any $k \in \I_{n-1}$.
If $q \neq 1$ is a primitive $n^{th}$ root of unity, then $(n-1)_q ! \neq 0$.

Furthermore, we have the following $q$-analogues for \emph{Pascal identity}:
for any $q \in \k^{\times}$, $i \in \N$ and $s \in \Z$, the equalities
	\begin{align} \label{6.5}
	{i \choose s}_q  =  {{i-1} \choose {s-1}}_q + q^{s}  {{i-1} \choose s}_q 
	\end{align}
and
	\begin{align} \label{6.7}
	{i \choose s}_q  =  {{i-1} \choose s}_q  + q^{i-s}  {{i-1} \choose {s-1}}_q
	\end{align}
hold.

\section{Preliminaries}\label{sec:preliminaries}

\subsection{Partial (Co)Actions}\label{Subsec:partial_actions}
Partial actions and partial coactions of bialgebras 
on algebras were introduced in \cite{caenepeel2008partial}.
However, we will use the definitions as they appear in \cite{Partial_representations, dilations}, since these later works care about symmetric conditions on partial (co)actions.
These conditions are important for some developments in the theory of partial (co)actions, such as partial (co)representations.

\medbreak

	A \emph{left partial action of a bialgebra $H$ on an algebra $A$} is a linear map $\cdot: H \otimes A \longrightarrow A$, denoted by $\ap (h \o a) = h \ap a$, such that
	\begin{enumerate}
		\item [(i)] $1_H \cdot a=a$;
		\item [(ii)] $h\cdot ab=(h_1\cdot a)(h_2\cdot b)$;
		\item [(iii)] $h\cdot(k\cdot a)=(h_1\cdot 1_A)(h_2k\cdot a)$,
	\end{enumerate}
hold for all $h,k\in H$ and $a,b\in A$.
In this case, $A$ is called a \textit{left partial $H$-module algebra}.
A left partial action is \emph{symmetric} if in addition we have
	\begin{enumerate}
		\item [(iv)] $h \cdot ( k \cdot a)=(h_1k \cdot a)(h_2 \cdot 1_A)$,
	\end{enumerate}
for all $h,k\in H$ and $a \in A$.
In this case, $A$ is called a \textit{symmetric left partial $H$-module algebra}.

\smallbreak

The definition of a \emph{right partial action} is given analogously.
Since throughout this work we deal only with left partial actions, from now on by a \emph{partial action} we mean a left partial action.

\smallbreak

It is clear that every $H$-module algebra is a symmetric partial $H$-module algebra. Moreover, a partial $H$-module algebra is an $H$-module algebra if and only if $h \cdot 1_A = \varepsilon(h) 1_A,$ for all $h \in H$.

\begin{ex}\cite[Proposition 4.10]{caenepeel2008partial}\label{lda_Ug} Let $\mathfrak{g}$ be a Lie algebra and $U(\mathfrak{g})$ its universal enveloping algebra.
	Then, every partial $U(\mathfrak{g})$-module algebra is a $U(\mathfrak{g})$-module algebra.
\end{ex}

Inspired by an example in \cite{partial_invariants}, we construct \emph{genuine partial actions}, \emph{i.e.}, partial actions that are not global.
\begin{ex}\label{exemplo_inicial} Let $G$ be a group and $N$ a subgroup of $G$.
	Consider the map $\delta_N : G \longrightarrow \k$ given by $\delta_N (g) = 1$ if $g \in N$, and $\delta_N(g) = 0$ otherwise.
	Then, for any algebra $A$, the linear map $\ap: \k G \o A \longrightarrow A$ given by $g \ap a = \delta_N(g)a$, for all $g \in G, a \in A$, defines a symmetric partial action of $\k G$ on $A$.
	If $N \neq G$, there exists an element $g \in G$ such that $g \notin N$, and for such an element $g \ap 1_A = 0 \neq 1_A = \varepsilon(g)1_A$, \emph{i.e.}, the map  $\cdot$ is a genuine partial action.
\end{ex}

The previous example can be generalized to produce a partial action of a bialgebra $H$ on any algebra $A$ through scalars.
Indeed, if the linear map $\cdot : H \o \k \longrightarrow \k$ is a partial action of a bialgebra $H$ on its base field $\k$, then the linear map $\rightharpoonup :H \o A \longrightarrow A$, given by $h \rightharpoonup a = (h \ap 1_\k)a,$ for all $h \in H, a \in A$, is a partial action of $H$ on $A$.

In this way, partial actions through scalars provide a large amount of examples of partial actions.
We dedicate Subsection \ref{properties} exclusively to such partial (co)actions.

\medbreak

	A \emph{right partial coaction of a bialgebra $H$ on an algebra $A$} is a linear map $\rho: A \longrightarrow A\otimes H$ such that:
	\begin{enumerate}
		\item [(i)] $[(Id_A \otimes \varepsilon_H) \circ \rho](a)=a$;
		\item [(ii)] $\rho(ab)=\rho(a)\rho(b)$;
		\item [(iii)] $[(\rho \otimes Id_H) \circ \rho](a) = (\rho(1_A) \otimes 1_H)([(Id_A \otimes \Delta_H) \circ \rho](a))$,
	\end{enumerate}
	for all $a, b \in A$.
	In this case, $A$ is called a \textit{right partial $H$-comodule algebra}.
	
	A right partial coaction is called \emph{symmetric} if satisfies additionally
	\begin{enumerate}
		\item[(iv)] $[(\rho \otimes Id_H) \circ \rho](a) = ([(Id_A \otimes \Delta_H) \circ \rho](a))(\rho(1_A) \otimes 1_H)$,
	\end{enumerate}
	for all $a \in A$.
	In this case, $A$ is called a \textit{symmetric right partial $H$-comodule algebra}.
We write simply $a^0 \o a^1$ to mean $\rho(a) = \sum_{i =1}^{n(a)} a_i \o h_i \in A \o  H$.

\medbreak

Similarly, we have the definition of a \emph{left partial coaction}.
Throughout this work we will deal only with right partial coactions, shortly called \emph{partial coactions}.

\smallbreak

It is a well-known fact that an $H$-comodule algebra is a symmetric partial $H$-comodule algebra. Moreover, a partial $H$-comodule algebra is an $H$-comodule algebra if and only if $\rho(1_A) = 1_A \otimes 1_H$. 

\smallbreak

The following example is the analogous of Example \ref{exemplo_inicial} for the setting of partial coactions, that is, $\k G$ coacts partially on any algebra $A$ through scalars.

\begin{ex} Let $G$ be a group, $N$ a finite subgroup of $G$ and suppose that $char(\k) \nmid |N|$.
Consider the element $z=\frac{1}{|N|}\sum_{g \in N}g \in \k G$.
	Then, for any algebra $A$, the linear map $\rho_N : A \longrightarrow A \o \k G$ given by $\rho_N (a) = a \o z$, for all $a \in A$, is a symmetric partial coaction of $\k G$ on $A$.
	If $N \neq \{1\}$, then $z \neq 1_{\k G}$ and so $\rho_N(1_A) = 1_A \o z \neq 1_A \o 1_{\k G}$.
	Therefore, if $N \neq \{1\}$, then $\rho_N$  is a \emph{genuine partial coaction}, \emph{i.e.}, $\rho_N$ is not a global coaction.
\end{ex}

The previous example can be generalized for any bialgebra $H$ as follows: if $\rho: \k \longrightarrow \k \o H$ is a partial coaction of $H$ on $\k$, then, for any algebra $A$, the linear map $\tilde{\rho} : A \longrightarrow A \o H$ given by $\tilde{\rho}(a) = (a \o 1_H) \rho(1_\k)$, for all $a \in A$, is a partial coaction of $H$ on $A$.

\smallbreak

To conclude this subsection, we consider restrictions and extensions of partial (co)actions.

\begin{rem}\label{restricoes}
	Let $A$ be an algebra, $H$ a bialgebra and $B \subseteq H$ a subbialgebra.
	If $\ap : H \o A \longrightarrow A$ is a (symmetric) partial action of $H$ on $A$, then $\ap |_{B \o A} : B \o A \longrightarrow A$ is a (symmetric) partial action of $B$ on $A$.
	If $\rho : A \longrightarrow A \o B$ is a (symmetric) partial coaction of $B$ on $A$, then $\tilde{\rho} : A \longrightarrow A \o H$, given by $\tilde{\rho}(a) = \rho(a)$ for all $a \in A$, is a (symmetric) partial coaction of $H$ on $A$.
\end{rem}

\subsection{Partial (co)actions on the base field}\label{properties}

In this subsection we deal with partial (co)actions of a bialgebra $H$ on its base field $\k$.
We exhibit all Hopf algebras whose partial (co)actions on their base fields are classified.
Moreover, some useful results to calculate such partial (co)actions are presented.

\medbreak

The following lemma characterizes partial actions of a bialgebra $H$ on its base field $\k$ through maps $\lda \in H^*$ satisfying some properties \cite{Alvares_Alves_Batista, glob_twisted_patial_hopf}.
This correspondence holds for more general settings, \emph{e.g.}, if $H$ is a weak bialgebra \cite{guris}.
\begin{lema}
	\label{k_mod_alg_parc}\label{eqn_lda}
	Let $H$ be a bialgebra and $\lda: H \longrightarrow \k$ a linear map. Then, $\lda$ defines a partial action of $H$ on $\k$ via $h \ap 1_\k = \lda (h)$, for all $h \in H$, if and only if $\lda(1_H)=1_\k$ and
	\begin{align}\label{eq}
	\lda(h)\lda(y)=\lda(h_1)\lda(h_2y)
	\end{align}
	holds for all $h, y \in H$.
Moreover, the partial action is symmetric if and only if $\lda$ satisfies additionally the condition
\begin{align}\label{eqn_lda_sim}
\lda(h)\lda(y)=\lda(h_1y)\lda(h_2),
\end{align}
for all $h, y \in H$.
\end{lema}

Using this one-to-one correspondence, from now on, by a partial action of a Hopf algebra $H$ on its base field $\k$ we mean a linear map $\lda \in H^*$ such that $\lda(1_H)=1_\k$ and satisfies the condition \eqref{eq}.

\smallbreak

	Observe that if $\lda \in H^*$ is a partial action of $H$ on $\k$, then $\lda$ is an idempotent element in the convolution algebra $H^*$.

\smallbreak

To the best of the authors' knowledge, the next examples are all the Hopf algebras that have partial actions on their base field classified in the literature.
Such results first appear in \cite{Alvares_Alves_Batista} and can be found in several later works  \cite{glob_twisted_patial_hopf, Partial_representations}.

	\begin{ex}[\bf{Partial actions of $\k G$ on $\k$}] \label{lda_kg}
	Let $G$ be a group.
	The partial actions of $\k G$ on $\k$ are in one-to-one correspondence with the subgroups of $G$.
	Indeed, if $N$ is a subgroup of $G$, then $\lda_N$ is a partial action of $\k G$ on $\k$, where $\lda_N: \k G \longrightarrow \k $ is defined by $\lda_N (g) = 1_\k$ if $g \in N$ and $\lda_N (g) = 0$ otherwise.
	Conversely, let $\lda: \k G \longrightarrow \k$ be a partial action of $\k G$ on $\k$.
	Then, the subset $\{g \in G \, : \, \lda(g) \neq 0 \} = \{g \in G : \lda(g) =1_\k \}$ is a subgroup of $G$.
	\end{ex}
	
	\begin{ex}[\bf{Partial actions of $(\k G)^*$ on $\k$}] \label{lda_kg_est}
	Let $G$ be a finite group.
Partial actions of $(\k G)^*$ on $\k$ are in correspondence with the set
$$\{ N \ : \ N \textrm{ is a subgroup of } G \textrm{ and } char(\k) \nmid |N| \}.$$
Indeed, let $\{g^* : g \in G\}$ be the dual basis of the canonical basis $\{g : g \in G\}$ of $\k G$.
For each subgroup $N$ of $G$ such that $char(\k) \nmid |N|$, the linear map $\lambda^*_N : (\k G)^* \longrightarrow \k$, given by
$\lda^*_N (g^*) = 1_\k/ |N|$ if $g \in N$ and $\lda^*_N (g^*) = 0$ otherwise, is a partial action of $(\k G)^*$ on $\k$.
Conversely, let $\lda$ be a partial action of $(\k G)^*$ on $\k$.
Then, the subset $M = \{g \in G \, : \, \lda(g^*) \neq 0 \}$ is a subgroup of $G$.
In this case, it is known that $\lda(g^*) = 1_\k / |M|$, for all $g \in M$, and so $\lda = \lda^*_M$.
\end{ex}

\begin{ex}[\bf{Partial actions of Sweedler's Hopf algebra on $\k$}]\label{lda_sweedler}
	Suppose $char(\k) \neq 2$ and consider the Sweedler's 4-dimensional Hopf algebra $\mathbb{H}_4$.
	Precisely, $\mathbb{H}_4$ is the Hopf algebra generated by a group-like element $g$ and an $(1,g)$-primitive element $x$ with relations $g^2=1, x^2=0$ and $xg =-gx$.
	The set $\{1, g, x, gx\}$ is a basis for $\mathbb{H}_4$.
	For each $\alpha \in \k$, the linear map $\lda_\alpha : \mathbb{H}_4 \longrightarrow \k$ given by $\lda_\alpha (1)=1_\k, \lda_\alpha(g)=0$ and $\lda_\alpha(x)=\lda_\alpha(gx)= \alpha$ is a partial action of $\mathbb{H}_4$ on $\k$.
	Furthermore, if $\lda$ is a partial action of $\mathbb{H}_4$ on $\k$, then $\lda=\varepsilon$ (global action) or $\lda=\lda_\alpha$, for some $\alpha \in \k$.
\end{ex}

\begin{ex}[\bf{Partial actions of $U(\mathfrak{g})$ on $\k$}] \label{lda_U_lie}
	Let $\mathfrak{g}$ be a Lie algebra and $U(\mathfrak{g})$ its universal enveloping algebra.
	As stated in Example \ref{lda_Ug}, the unique partial action of $U(\mathfrak{g})$ on $\k$ is the global action given by its counit.
\end{ex}

\noindent \textbf{Remark.} Using condition \eqref{eqn_lda_sim}, it is straightforward to check that every partial action presented in the previous examples are symmetric.

\medbreak

Partial coactions of a bialgebra on its base field have the following characterization, adapted from \cite[Example 2.9]{caenepeel2008partial} to symmetric setting that we are dealing.
\begin{lema}\label{k_comod_alg_parc}
	Let $H$ be a bialgebra. Then, $\k$ is a partial $H$-comodule algebra if and only if there is an idempotent element $z \in H$ such that $\varepsilon_H (z)=1_\k$ and
	\begin{align}\label{eqn_z} z \otimes z = (z \otimes 1_H)\Delta_H(z).
	\end{align}
	Moreover, $\k$ is a symmetric partial $H$-comodule algebra if and only if the idempotent element $z$ satisfies additionally 
	\begin{align}\label{eqn_z_sim}
	z \o z = \Delta_H(z)(z \otimes 1_H).
	\end{align}
\end{lema}

For $ \rho : \k \longrightarrow \k \o H$ a linear map,
there exists a unique element $z \in H$ such that $\rho(1_\k) = 1_\k \o z$, and conversely each element $z \in H$ defines a linear map in this way.
We emphasize this fact denoting $\rho$ by $\rho_z$.  
Hence, Lemma \ref{k_comod_alg_parc} says that the map $\rho_z$ is a partial coaction of $H$ on $\k$ if and only if the element $z \in H$ satisfies $\varepsilon_H (z)=1_\k$ and \eqref{eqn_z}.
The partial coaction $\rho_z$ is symmetric if and only if \eqref{eqn_z_sim} also holds.
Thus, we will say that an element $z \in H$ is a partial coaction of $H$ on $\k$ to mean that $\rho_z$ is a partial coaction  of $H$ on $\k$.
Moreover, $\varepsilon_H (z)=1_\k$ and \eqref{eqn_z} imply that $z$ is an idempotent element.

\smallbreak

	Let $H$ be a finite-dimensional semisimple Hopf algebra. Then, there exists at least one element $ z \in H$ such that $\rho_z$ is a genuine partial coaction of $H$ on $\k$ \cite[Example 2.9]{caenepeel2008partial}.
	Indeed, it is enough to consider the right integral element $z$ such that $\varepsilon_H(z)=1_\k$.

\smallbreak

Now, using Lemma \ref{k_comod_alg_parc}, some examples are presented.

	\begin{ex}[\bf{Partial coactions of $U(\mathfrak{g})$ on $\k$}] \label{z_U_g}
	Let $\mathfrak{g}$ be a Lie algebra and $U(\mathfrak{g})$ its universal enveloping algebra.
	The unique partial coaction of $U(\mathfrak{g})$ on $\k$ is the global one. 
	Indeed, suppose that $z \in U(\mathfrak{g})$ is a partial coaction of $U(\mathfrak{g})$ on $\k$.
	Since $U(\mathfrak{g})$ is a domain (see, for instance, \cite[Theorem V.3.6]{Jacobson_lie_algebras}) and $z$ is an idempotent element satisfying $\varepsilon(z)=1_\k$, it follows that $z=1$ and therefore the global coaction.
\end{ex}

Let $H$ be a bialgebra.
	By Lemma \ref{k_comod_alg_parc} and \cite[Proposition 4.12]{glauber_e_fefi}, $\k$ is a right partial $H$-comodule algebra if and only if $\k$ is a left partial $H$-comodule coalgebra (\emph{cf.} \cite[Definition 6.1]{dual_constructions}).
Thus, all partial coactions of $\k G$ on $\k$ are given in \cite[Example 4.13]{glauber_e_fefi}, as below.
	\begin{ex}[\bf{Partial coactions of $\k G$ on $\k$}] \label{z_kg} 
	Let $G$ be a group and consider $z = \sum_{g \in N} \alpha_g g \in \k G$, where $N$ is a finite subset of $G$ and $0 \neq \alpha_g \in \k$, for all $g \in N$.
	Then, $z$ is a partial coaction of $\k G$ on $\k$ if and only if $N$ is a subgroup of $G$, $char(\k) \nmid |N|$ and $\alpha_g = 1/|N|$, for all $g \in N$.
	\end{ex}

\begin{ex}[\bf{Partial coactions of $(\k G)^*$ on $\k$}] \label{z_kg_est}
	Let $G$ be a finite group and consider $z = \sum_{g \in N}\alpha_g g^* \in (\k G)^*$, where $N$ is a subset of $G$ and $0 \neq \alpha_g \in \k$, for all $g \in N$.
	Then, $z$ is a partial coaction of $(\k G)^*$ on $\k$ if and only if $N$ is a subgroup of $G$ and $\alpha_g = 1$, for all $g \in N$.
	\end{ex}

\begin{ex}[\bf{Partial coactions of $\mathbb{H}_4$ on $\k$}] \label{z_sweedler}
	Assume $char(\k) \neq 2$.
	In \cite[Example 2.10]{caenepeel2008partial}, a family of genuine partial coactions of the Sweedler's 4-dimensional Hopf algebra $\mathbb{H}_4$ on its base field $\k$ is introduced.
	Precisely, for any $\alpha \in \k$, $z_\alpha = \frac{1+g}{2}-\alpha gx$ is a partial coaction of $\mathbb{H}_{4}$ on $\k$. 
\end{ex}

\noindent \textbf{Remark.} Using condition \eqref{eqn_z_sim}, it is easy to check that all partial coactions presented in the previous examples are symmetric.

\medbreak

The concepts of partial actions and partial coactions of a Hopf algebra on an algebra are dually related \cite{enveloping, Globalization}, as in the global case. 
Using the characterizations given in lemmas \ref{k_mod_alg_parc} and \ref{k_comod_alg_parc}, we state below such a dual relation in the particular case when $H$ is a finite-dimensional Hopf algebra (co)acting partially on its base field $\k$. 
For the general case see \cite[Corollary 1]{enveloping}.

\begin{prop}\label{dual_corpo}
	Let $H$ be a finite-dimensional Hopf algebra. Then, 
	$\k$ is a (symmetric) partial $H$-module algebra via $\lda \in H^*$ if and only if $\k$ is a (symmetric) partial $H^*$-comodule algebra via $\rho_\lda$.
\end{prop}

In others words, the above result says that, if $H$ is a finite-dimensional Hopf algebra, then the map $\lda \in H^*$ is a (symmetric) partial action of $H$ on $\k$ if and only if $\lda \in H^*$ is a (symmetric) partial coaction of $H^*$ on $\k$.

From Example \ref{lda_kg} and Proposition \ref{dual_corpo}, one easily obtain Example \ref{z_kg_est}.
Besides, since the Sweedler's 4-dimensional Hopf algebra $\mathbb{H}_{4}$ is a self-dual Hopf algebra, it follows by Example \ref{lda_sweedler} and Proposition \ref{dual_corpo} that Example \ref{z_sweedler} characterizes all partial coactions of $\mathbb{H}_{4}$ on $\k$, \emph{i.e.}, if $z \in \mathbb{H}_{4}$ is a partial coaction of $\mathbb{H}_{4}$ on $\k$, then $z=1$ (global) or $z = z_\alpha$, for some $\alpha \in \k$, as given in Example \ref{z_sweedler}.

\medbreak

Now let us recall Remark \ref{restricoes} with suitable data: let $H$ be a Hopf algebra and $B= \k G(H)$ its Hopf subalgebra of the group-like elements. 
If $\lda : H \longrightarrow \k$ is a partial action of $H$ on $\k$, then $\lda |_{\k G(H)} : \k G(H) \longrightarrow \k$ is a partial action of the group algebra $\k G(H)$ on $\k$ and so, by Example \ref{lda_kg}, there exists a subgroup $N$ of $G(H)$ such that $\lda |_{\k G(H)} = \lda_N$.
Besides, for any finite subgroup $N$ of $G(H)$ such that $char(\k) \nmid |N|$, the idempotent element $z_N = \frac{1}{|N|} \sum_{g \in N} g \in \k G(H)$ is a partial coaction of $\k G(H)$ on $\k$ (see Example \ref{z_kg}), and then $z_N$ also is a partial coaction of $H$ on $\k$.

\medbreak

The evaluation of a partial action on its base field in some elements is presented below.

\begin{prop}\label{propriedades_lda}
	Let $H$ be a Hopf algebra and $\lda: H \longrightarrow \k$ a partial action of $H$ on $\k$.
	Consider the elements $g, t, x \in H$ such that $g, t \in G(H)$ and $x$ is a $(g,t)$-primitive element.
	Then:
	\begin{itemize}
		\item[(i)] if $\lda(g)=1$, then $\lda(gu)=\lda(u)$, for all $u \in H$;
		\item[(ii)] if $\lda(g)=\lda(t)$, then $\lda(x)=0$;
		\item[(iii)] if $\lda(x)=0$ and $\lda(t)=1$, then $\lda(xu)=0$, for all $u \in H$.
	\end{itemize}
\end{prop}
\begin{proof}
	It follows directly from the application of \eqref{eq} for good choices of elements of $H$.
\end{proof}

\section{Partial (Co)Actions of the Taft Algebra $T_n(q)$ on $\k$}\label{sec:Taft}
In this section we calculate all partial (co)actions of the Taft algebra $T_n(q)$ on its base field $\k$.
Next we recall this family of finite-dimensional Hopf algebras introduced by Taft in \cite{Taft}.

\medbreak

Let $n \geq 2$ be an integer and suppose that $\k$ contains a primitive $n^{th}$ root of unity $q$.
In particular $char(\k) \nmid n$.
The \emph{Taft algebra of order $n$}, or shortly \emph{Taft algebra}, here denoted by $T_n(q)$, has the following structure:
as algebra it is generated over $\k$ by two elements $g$ and $x$ with relations $g^n=1,$ $x^n =0$ and $xg = q gx$.
Thus, the set $\{g^ix^j \ : \  i, j \in \I_{0,n-1} \}$ is the canonical basis for $T_n(q)$ and so $dim_\k(T_n(q)) = n^{2}$.
The coalgebra structure of $T_n(q)$ is induced by setting $g$ a group-like element and $x$ an $(1,g)$-primitive element, that is, $\Delta (g) = g \o g$, $\Delta (x) = x \o 1 + g \o x$, $\varepsilon (g) = 1$ and $\varepsilon (x) = 0$.
In general, the comultiplication map is $\Delta (g^ix^j) = \sum_{\ell=0}^j {j \choose \ell}_q  g^{i+\ell}x^{j-\ell} \o g^ix^\ell$, for $i, j \in \I_{0,n-1}$.
To complete the Hopf algebra structure of $T_n(q)$, the antipode map $S$ is defined by $S(g) =g^{n-1}$ and $S(x) = -g^{n-1}x.$
Furthermore, note that the group of group-like elements is $G( T_n(q) ) = \langle g \rangle = \{ 1, g, \cdots, g^{n-1} \}$, \emph{i.e.}, $ G( T_n(q) ) =  C_n.$

\smallbreak

It is a well-known fact that Taft algebras are self-dual Hopf algebras.
This property will allow us to calculate partial coactions from partial actions of the Taft algebra on its base field.

When $n=2$, the Taft algebra $T_2(-1)$ is exactly the Sweedler's four-dimensional Hopf algebra $\mathbb{H}_{4}$ (see Example \ref{lda_sweedler}).

\subsection{Partial Actions}\label{Subsec:acoes_taft}
In this subsection we shall see that there exist two families of genuine partial actions of $T_n(q)$ on $\k$, one parameterized by $\k$ and another by non-trivial subgroups of $C_n$.

\medbreak

Let $\lda \in T_n(q)^*$ be a partial action of $T_n(q)$ on $\k$. Then, the restriction of $\lda$ to the Hopf subalgebra $\k G(T_n(q)) = \k C_n$ is a partial action of $\k C_n$ on $\k$.
 
In this case, we know  that $\lda{|_{\k G(T_n(q) ) }}$ is parameterized by a subgroup $N$ of $C_n,$ that is, $\lda{|_{\k G(T_n(q) ) }} = \lda_N$.
Furthermore, the associated subgroup is $N=\{g \in G(T_n(q)) \ : \ \lda(g)=1_\k \}$.
See Example \ref{lda_kg}.

\medbreak

For what follows, we define:
\begin{enumerate}
	\item for each $\alpha \in \k$, the linear map $\lda_\alpha: T_n(q) \longrightarrow \k$ given by
	\begin{equation}\label{def_lda_alpha}
	\lda_{\alpha}(g^{n-i}x^{j}) = q^{\frac{i(i+1)}{2}} {j \choose i}_q  (-1)^i \alpha^j,
	\end{equation}
	for all $i,j \in \I_{0,n-1}$.
	In particular, $\lda_\alpha (x) = \alpha$.
	
	\smallbreak
	
	\item for each subgroup $N$ of $C_n$, the linear map $\lda_N^0: T_n(q) \longrightarrow \k$ defined as 
	\begin{equation}\label{def_lda_N_zero}
	\lda_N^0(g^ix^j) = \delta_{j,0}\lda_N(g^i).
	\end{equation}
	Note that $\lda_{C_n}^0 = \varepsilon$ and $\lda_{\{1\}}^0 = \lda_0$ (as in \eqref{def_lda_alpha}).
	In terms of the dual basis, we have
\begin{equation}\label{lda_N_base_dual}
		\lda_N^0 = \sum_{h \in N} h^*.
\end{equation}
\end{enumerate}

So, here is the main result of this subsection:
\begin{teo}\label{taft}
	Let $\lda:T_n(q) \longrightarrow \k$ be a linear map. 
	Then, $\lda$ is a partial action of $T_n(q)$ on $\k$ if and only if $\lda = \varepsilon$ (global action), $\lda=\lda_N^0$ or $\lda=\lda_\alpha$, where 
	$\alpha \in \k$ and $N$ is a non-trivial subgroup of $G(T_n(q) )$.
\end{teo}

For the proof of Theorem \ref{taft}, we need some preliminary results, which we develop below.
\begin{lema}\label{lema1}
	Let $t,k \in \N_0$ and $q \in \k^{\times}$.
	Then, for each  $i \in \N_0$, 
	$$\sum_{s=0}^{i} {i \choose s}_q {{i+t-s} \choose {i+k}}_q (-1)^s q^{sk +\frac{s(s+1)}{2}}={t \choose k}_q.$$
\end{lema}

\begin{proof}
	We prove by induction on $i$. The statement holds for $i=0$. For $i > 0$, note that 
	\begin{align*}
	& \sum_{s=0}^{i} {i \choose s}_q {{i+t-s} \choose {i+k}}_q (-1)^s q^{sk +\frac{s(s+1)}{2}} \\
	= & {{i+t} \choose {i+k}}_q + \sum_{s=1}^{i-1} {i \choose s}_q {{i+t-s} \choose {i+k}}_q  (-1)^s q^{sk +\frac{s(s+1)}{2}} \\
	+ & {{t} \choose {i+k}}_q (-1)^i q^{ik +\frac{i(i+1)}{2}}.
	\end{align*}
	Since ${i \choose s}_q  = {{i-1} \choose s}_q +{{i-1} \choose {s-1}}_q q^{i-s}$ by \eqref{6.7}, we continue the above equality as equal to
	\begin{align*}
	& {{i+t} \choose {i+k}}_q + \sum_{s=1}^{i-1} {{i-1} \choose s}_q {{i+t-s} \choose {i+k}}_q (-1)^s q^{sk +\frac{s(s+1)}{2}} \\
	+ & \sum_{s=1}^{i-1} {{i-1} \choose {s-1}}_q {{i+t-s} \choose {i+k}}_q (-1)^s q^{i-s} q^{sk +\frac{s(s+1)}{2}}  \\
	+ & {{t} \choose {i+k}}_q (-1)^i q^{ik +\frac{i(i+1)}{2}} \\
	= & {{i+t} \choose {i+k}}_q + \sum_{s=1}^{i-2} {{i-1} \choose s}_q {{i+t-s} \choose {i+k}}_q (-1)^s q^{sk +\frac{s(s+1)}{2}} \\
	+ & {{t+1} \choose {i+k}}_q (-1)^{i-1} q^{(i-1)k +\frac{(i-1)i}{2}}  -  {{i+t-1} \choose {i+k}}_q q^{k+i} \\
	+ & \sum_{s=2}^{i-1} {{i-1} \choose {s-1}}_q {{i+t-s} \choose {i+k}}_q (-1)^s q^{k+i} q^{(s-1)k +\frac{(s-1)s}{2}} \\
	+ & {{t} \choose {i+k}}_q (-1)^i q^{ik +\frac{i(i+1)}{2}} \\
	 = & {{i+t} \choose {i+k}}_q + \sum_{s=1}^{i-2} {{i-1} \choose s}_q {{i+t-s} \choose {i+k}}_q (-1)^s q^{sk +\frac{s(s+1)}{2}}  \\
	+ & {{t+1} \choose {i+k}}_q (-1)^{i-1} q^{(i-1)k +\frac{(i-1)i}{2}} -  {{i+t-1} \choose {i+k}}_q q^{k+i}  \\
	- & \sum_{s=1}^{i-2} {{i-1} \choose {s}}_q {{i+t-s-1} \choose {i+k}}_q (-1)^s q^{k+i} q^{sk +\frac{s(s+1)}{2}}  \\
	+ & {{t} \choose {i+k}}_q (-1)^i q^{ik +\frac{i(i+1)}{2}}  \\
	= & \left[ {{i+t} \choose {i+k}}_q -  {{i+ t -1} \choose {i+k}}_q q^{k+i}  \right] \\
	+ & \sum_{s=1}^{i-2} {{i-1} \choose s}_q \left[ {{i+t-s} \choose {i+k}}_q  -  q^{k+i}{{i+t-s-1} \choose {i+k}}_q \right] (-1)^s q^{sk +\frac{s(s+1)}{2}} \\
	+ & \left[ {{t+1} \choose {i+k}}_q -  {{t} \choose {i+k}}_q q^{k+i}  \right]  (-1)^{i-1} q^{(i-1)k +\frac{(i-1)i}{2}}.
	\end{align*}
	Now, by \eqref{6.5}, we replace ${{i+t-s} \choose {i+k}}_q  -  q^{k+i}{{i+t-s-1} \choose {i+k}}_q = {{i+t-s-1} \choose {k+i-1}}_q,$ $ {{i+t} \choose {i+k}}_q - {{i+t-1} \choose {i+k}}_q  q^{k+i} = {{i+t-1} \choose {i+k-1}}_q$ and $  {{t+1} \choose {i+k}}_q -  {{t} \choose {i+k}}_q q^{k+i}  = {{t} \choose {i+k-1}}_q$, to obtain
	\begin{align*}
	&{{i+t-1} \choose {i+k-1}}_q + \sum_{s=1}^{i-2} {{i-1} \choose s}_q {{i+t-s-1} \choose {k+i-1}}_q (-1)^s q^{sk +\frac{s(s+1)}{2}} \\
	+ &  {{t} \choose {i+k-1}}_q (-1)^{i-1} q^{(i-1)k +\frac{(i-1)i}{2}} \\
	= & \sum_{s=0}^{i-1} {{i-1} \choose s}_q {{i+t-s-1} \choose {k+i-1}}_q (-1)^s q^{sk +\frac{s(s+1)}{2}},
	\end{align*}
and this concludes the induction step.
\end{proof}

\begin{lema}\label{prodqbinom}
	Let $i,j,\ell \in \Z$ such that $0 \leq \ell \leq i \leq j$.
	Then, 
	$${ j \choose \ell }_q { {j-\ell} \choose {i-\ell} }_q = { j \choose i }_q { i \choose \ell }_q .$$
\end{lema}
\begin{proof}
It is clear if $(j)_q!=0$.
	Otherwise,
	\begin{align*}
	{ j \choose \ell }_q { {j-\ell} \choose {i-\ell} }_q & =  \frac{(j)_q!}{(j-\ell)_q! \, (\ell)_q!} \, \frac{(j-\ell)_q!}{((j-\ell)-(i-\ell))_q! \, (i-\ell)_q!} \\
	& = \frac{(j)_q!}{(\ell)_q! \,(j-i)_q! \, (i-\ell)_q!} \\
	& = \frac{(j)_q!}{(i)_q! \,(j-i)_q!} \, \frac{(i)_q!}{(\ell)_q! \, (i-\ell)_q!} \\
	& = { j \choose i }_q { i \choose \ell }_q. \qedhere
	\end{align*}
\end{proof}

\begin{prop}\label{lda_alpha_eh_acao_parcial}
	For any $\alpha \in \k$, the linear map $\lda_\alpha: T_n(q) \longrightarrow \k$ given by $\lda_{\alpha}(g^{n-i}x^{j}) = q^{\frac{i(i+1)}{2}} {j \choose i}_q  (-1)^i \alpha^j,$ for $ i,j \in \I_{0,n-1}$, is a partial action of $T_n(q)$ on $\k$.
\end{prop}
\begin{proof}
	It is enough to check condition \eqref{eq} considering $h=g^{n-i}x^{j}$ and $y=g^{n-k}x^{t}$, for $i,j,k,t \in \I_{0,n-1}$, that is, if the following equality holds:
	\begin{align*}
	\lda_\alpha(g^{n-i}x^{j})\lda_\alpha(g^{n-k}x^{t}) = \sum_{\ell=0}^{j} {j \choose \ell}_q q^{-k \ell} \lda_\alpha(g^{n-(i-\ell)} x^{j-\ell})\lda_\alpha(g^{n-(i+k)} x^{\ell+t}).
	\end{align*}
	
	If $i >j$, then the above equality holds, since $\lda_\alpha(g^{n-i}x^{j}) = 0$ and also $ \lda_\alpha(g^{n-(i-\ell)} x^{j-\ell})=0$, for all $\ell \in \I_{0,j}$.
	
	Now we assume $i \leq j$.
	Note that for $\ell \in \I_{i+1,j}$, we have $0 < n-(\ell-i) < n$ and then $\lda_\alpha(g^{n-(i-\ell)} x^{j-\ell}) = \lda_\alpha(g^{n-(n-(\ell-i))} x^{j-\ell}) = 0$, since  $n - (\ell-i) > (j-\ell)$.
	Thus, condition \eqref{eq} can be rewritten as
	\begin{align*}
	\lda_\alpha(g^{n-i}x^{j})\lda_\alpha(g^{n-k}x^{t}) = \sum_{\ell=0}^{i} {j \choose \ell}_q q^{-k \ell} \lda_\alpha(g^{n-(i-\ell)} x^{j-\ell})\lda_\alpha(g^{n-(i+k)} x^{\ell+t}).
	\end{align*}
	Using the definition of $\lda_\alpha$ in the above equality we obtain
	\begin{align*}
	&(-1)^iq^{\frac{i(i+1)}{2}} {j \choose i}_q \alpha^j (-1)^k q^{\frac{k(k+1)}{2}} {t \choose k}_q \alpha^{t} \\
	= & \sum_{\ell=0}^{i} {j \choose \ell}_qq^{-k\ell} (-1)^{i-\ell} q^{\frac{(i-\ell)(i-\ell+1)}{2}} { j-\ell \choose i-\ell}_q  \alpha^{j-\ell} \\
	\times & (-1)^{i+k} q^{\frac{(i+k)(i+k+1)}{2}} {{\ell+t} \choose {i+k}}_q \alpha^{\ell+t},
	\end{align*}
	and then using Lemma \ref{prodqbinom} and simplifying, we get
	\begin{align*}
	\alpha^{j+t} {t \choose k}_q = & \ \alpha^{j+t} \sum_{\ell=0}^{i} {i \choose \ell}_q {{\ell+t} \choose {i+k}}_q (-1)^{i-\ell} q^{k(i-\ell)} q^{\frac{(i-\ell)(i-\ell+1)}{2}}.
	\end{align*}
	Since ${i \choose \ell}_q ={i \choose {i-\ell}}_q$ for all $\ell \in \I_{0,i}$, we rewrite the above equality as
	\begin{align*}
	\alpha^{j+t} {t \choose k}_q = \alpha^{j+t} \sum_{s=0}^{i} {i \choose {i-s}}_q {{i+t-s} \choose {i+k}}_q (-1)^{s} q^{ks} q^{\frac{s(s+1)}{2}},
	\end{align*}
	which is true by Lemma \ref{lema1}.
	
	Therefore, $\lda_\alpha$ is a partial action of $T_n(q)$ on $\k$.
\end{proof}

\begin{prop}\label{lda_N_zero_eh_acao_parcial}
	For each subgroup $N$ of $C_n$, consider the linear map $\lda_N^0: T_n(q) \longrightarrow \k$ given by $\lda_N^0(g^i x^j) = \delta_{j,0} \lda_N (g^i)$, $ i,j \in \I_{0,n-1}$.
	Then, $\lda_N^0$ is a partial action of $T_n(q)$ on $\k$.
\end{prop}
\begin{proof}
	By Lemma \ref{eqn_lda}, we need to verify condition \eqref{eq}, \emph{ i.e.}, if
	\begin{equation}\label{verific_N}
	\lda_N^0(h)\lda_N^0(y)=\lda_N^0(h_1)\lda_N^0(h_2y)
	\end{equation}
	 holds for all $h,y \in T_n(q)$.
	
	It is clear if $N=C_n$ or $N=\{1\}$, since $\lda_{C_n}^0=\varepsilon$ and $\lda_{\{1\}}^0=\lda_0$ (see Proposition \ref{lda_alpha_eh_acao_parcial}).
	
	Now suppose $N$ a non-trivial subgroup of $C_n$, that is, $\{1\} \neq N \neq C_n$.
	In this case, there exist integers $k, \ell \in \I_{2,n-1}$ such that $k \ell = n$ and $N= \langle g^k \rangle = \{ 1, g^k, g^{2k}, ..., g^{(\ell-1)k} \}$.
	Thus, $\lda(g^{ik}) = 1$ and $\lda(g^{ik+j})=0,$ for all $i \in \Z$ and $j \in \I_{k-1}$. 
	
	First, consider $h=g^{ik+j}$ with $i \in \Z$ and $j \in \I_{0,k-1}$. If $j \in \I_{k-1}$, then \eqref{verific_N} holds trivially.
	Otherwise $j=0$ and then \eqref{verific_N} results in
	$\lda_N^0(y)=\lda_N^0(g^{ik}y).$
	For $y=g^t x^s$, the previous equality results in $\delta_{s,0}\lda_N(g^t)=\delta_{s,0}\lda_N(g^{ik+t})$.
	Since $g^{ik} \in N$, it follows that $g^{ik+t} \in N$ if and only if $g^{t} \in N$, and so \eqref{verific_N} holds.
	
	Now, consider $h= g^ix^j$, with $i \in \Z$ and $j \in \I_{k-1}$.
	Then, \eqref{verific_N} becomes $$\lda_N^0(g^ix^j)\lda_N^0(y) = \sum_{\ell=0}^j {j \choose \ell}_q  \lda_N^0(g^{i+\ell}x^{j-\ell}) \lda_N^0(g^ix^\ell y),$$ for $y \in T_n(q)$.
	Since $\lda_N^0(g^ix^j)=\delta_{j,0} \lda_N (g^i)=0$ and $\lda_N^0(g^{i+\ell}x^{j-\ell})=\delta_{j-\ell,0} \lda_N (g^{i+\ell}) = 0$, for all $\ell \in \I_{0,j-1}$,  the previous equality results in 
	$$0 = \lda_N^0(g^{i+j}) \lda_N^0(g^ix^jy).$$ This latter equality holds because $ \lda_N^0(g^ix^jy)=0$, for all $y \in T_n(q)$.
	Indeed, for $y=g^kx^t$, $\lda_N^0(g^ix^jg^kx^t)= q^{jk}\lda_N^0(g^{i+k}x^{j+t})=q^{j k} \delta_{j+t,0}\lda_N(g^{i+k}) = 0$, since $t+j \geq 1$.
	
	Thus, the linear map $\lda_N^0$ as defined is a partial action of 	$T_n(q)$ on $\k$.
\end{proof}
 
\medbreak

Now we can prove Theorem \ref{taft}, announced at the beginning of this subsection.

\begin{proof}[Proof of Theorem \ref{taft}]\label{demonstracao_taft}
First, Propositions \ref{lda_alpha_eh_acao_parcial} and \ref{lda_N_zero_eh_acao_parcial} ensure that $\lda_\alpha$ and  $\lda_N^0$ are partial actions of $T_n(q)$ on $\k$.

Conversely, suppose $\lda: T_n(q) \longrightarrow \k $ a partial action.
Then, $\lda|_{\k G(T_n(q))} = \lda_N$, where $N =\{h \in G(T_n(q)) \, : \, \lda(h)=1_\k \}$ is a subgroup of $G(T_n(q))$.
	We will divide $N$ into three possibilities and deal each case separately: $N = G( T_n(q) )$, $\{1\} \neq N \neq G(T_n(q) )$ and $N = \{1\}$.
		
	\smallbreak
	
	\underline{\textbf{$1^{st}$ case: $N=G(T_n(q))$}}
	
	By assumption, we have $\lda(g^i) =1$, for all $ i \in \I_{0,n-1}$.
	Now, using Proposition \ref{propriedades_lda}, we can conclude that $\lda = \varepsilon$.
	Indeed, first $\lda(g^i x^j) = \lda(x^j)$, for all $i, j \in \I_{n-1}$, by item (i).
	Besides, it follows by item (ii) that $\lda(x)=0$ and then, by item (iii), $\lda(x^j)=0$ for all $ j \in \I_{n-1}$.
	
	\smallbreak

	\underline{\textbf{$2^{nd}$ case: $\{1\} \neq N \neq G(T_n(q) )$}}
	
	Since $N$ is a non-trivial subgroup of $G(T_n(q)) = C_n$,
	there exist integers $k, \ell \in \I_{2,n-1}$ such that $k \ell = n$ and $N= \langle g^k \rangle = \{ 1, g^k, g^{2k}, ..., g^{(\ell-1)k} \}$.
	In this case $\lda(g^{ik}) = 1$ and $\lda(g^{ik+j})=0,$ for all $i \in \Z$ and $j \in \I_{k-1}$. 
	
	First, note that  $\lda(g^ix)=0$ for all $i \in \Z$.
	Indeed, by Proposition \ref{propriedades_lda} (ii) we obtain $\lda(g^{j}x)=0$ for $j \in \I_{k-2}$.
	Since $\lda(g^{ik}y)=\lda(y)$ for all $y \in T_n(q)$, by Proposition \ref{propriedades_lda} (i),
	we only need to check that $\lda(x)=0$ and $\lda(g^{k-1}x)=0$.
	By Lemma \ref{eqn_lda}, the equality
	\begin{align}\label{0}
	\lda(g^{k-1}x) \lda(y) = \lda(g^{k-1}x)\lda(g^{k-1}y) + \lda(g^{k-1}xy)
	\end{align}
	holds for all $y \in T_n(q)$.
	For $y=g$ and $y=g^{k+1}$ in the above equation, we obtain $\lda(g^{k-1}x)= -q \lda(x)$ and $\lda(g^{k-1}x) = -q^{k+1} \lda(x)$, respectively.
	From these equalities $\lda(x)= \lda(g^{k-1}x)=0$, since $q$ is a primitive $n^{th}$ root of unity and $ k \in \I_{2,n-1}$. 

	Now we observe that \eqref{0} is reduced to $0 = \lda(g^{k-1}xy)$, for all $y \in T_n(q)$, and therefore $\lda(g^ix^j)=0$, for all $i,j \in \Z$, $j \geq 1$.
	
	Hence, $\lda = \lda_N^0$ as in \eqref{def_lda_N_zero}.
	
	\smallbreak
	
	\underline{\textbf{$3^{rd}$ case: $N=\{1\}$}}
	
	In this case, we have $\lda(1) = 1$ and $\lda(g^i)=0,$ for all $i \in \I_{n-1}$. 
	First, by Proposition \ref{propriedades_lda} (ii), we obtain $\lda(g^ix)=0$, for all $ i \in \I_{n-2}$.
	Thus, it remains to know $\lda(x)$, $\lda(g^{n-1}x)$ and $\lda(g^ix^j)$, for $ i \in \I_{0,n-1}$ and $j \in \I_{2,n-1}$.
	By Lemma \ref{eqn_lda}, the following equation holds for all $y \in T_n(q)$:
	\begin{align}\label{nova}
	\lda(g^{n-1}x)\lda(y)= \lda(g^{n-1}x)\lda(g^{n-1}y)+ \lda(g^{n-1}xy).
	\end{align}
	For $y=g$ in the above equation, we get $\lda(g^{n-1}x) = -q \lda(x)$, and then we rewrite \eqref{nova} as
	\begin{align}\label{prin}
	-q \lda(x)\lda(y)=-q \lda(x)\lda(g^{n-1}y)+ \lda(g^{n-1}xy).
	\end{align}
	The previous equality is sufficient to determine the values $\lda(g^ix^j)$, for all $i \in \I_{0,n-1}$ and $j \in \I_{2,n-1}$, in terms of $\lda(x)$.
	But, since $g^n=1$, we will determine $\lda(g^{n-i}x^j)$ instead of $\lda(g^ix^j)$.
	
	\underline{\textbf{Claim:}} The following equality holds for all $i,j \in \I_{0,n-1}$:
	\begin{align}\label{acaotaft}
	\lda(g^{n-i}x^{j}) = (-1)^i q^{\frac{i(i+1)}{2}} {j \choose i}_q \lda(x)^j.
	\end{align}
	We prove by induction on $j$. Notice that the above equality holds for $j=0,1$. Assume $j \geq 0$.
	For $y= g^{n-(i-1)}x^j$ in  \eqref{prin}, we have
	\begin{align*}
	-q \lda(x)\lda( g^{n-(i-1)}x^j)=-q \lda(x)\lda(g^{n-i}x^j)+ q^{1-i}\lda(g^{n-i}x^{j+1}),
	\end{align*}
	and then
	\begin{align*}
	\lda(g^{n-i}x^{j+1}) =q^i\lda(x)\left( \lda(g^{n-i}x^j)-\lda( g^{n-(i-1)}x^j)\right).
	\end{align*}
	Using the induction hypothesis, we obtain
	\begin{align*}
	\lda(g^{n-i}x^{j+1}) = & q^i\lda(x)\left( (-1)^i q^{\frac{i(i+1)}{2}} {j \choose i}_q -(-1)^{i-1} q^{\frac{(i-1)i}{2}} {j \choose {i-1}}_q \right )  \lda(x)^j,
	\end{align*}
	that is,
	\begin{align*}
	\lda(g^{n-i}x^{j+1}) = & (-1)^iq^{\frac{i(i+1)}{2}} \lda(x)^{j+1} \left(  q^i {j \choose i}_q + {j \choose {i-1}}_q \right ).
	\end{align*}
	Thus, with \eqref{6.5}
	we conclude the induction step and consequently the proof of the claim.
	
		Therefore, $\lda = \lda_\alpha$ as in \eqref{def_lda_alpha}, where $\alpha = \lda(x)$.
\end{proof}

Let ${p_1}^{\gamma_1}{p_2}^{\gamma_2} \cdots {p_k}^{\gamma_k}$ be the prime factorization of the integer $n$.
Since $C_n$ has $(\gamma_1+1)(\gamma_2+1)\cdots (\gamma_k+1)$ subgroups and each subgroup $N$ of $C_n$ defines a family of partial actions of $T_n(q)$ on $\k$ (usually containing only one element, except for $N=\{1\}$), we obtain the following corollary.
\begin{cor}
	If ${p_1}^{\gamma_1}{p_2}^{\gamma_2} \cdots {p_k}^{\gamma_k}$ is the prime factorization of the integer $n$, then the Taft algebra $T_n(q)$ has $(\gamma_1+1)(\gamma_2+1)\cdots (\gamma_k+1)$ families of partial actions on $\k$.
\end{cor}

In the next result, the partial action $\lda_\alpha$ of $T_n(q)$ on $\k$, $\alpha \in \k$, is computed for some elements.

\begin{cor} Consider the partial action $\lda_\alpha : T_n(q) \rightarrow \k$, $\alpha \in \k$. Then
	\begin{enumerate}
		\item[$(i)$] $\lda_\alpha(x^j) = \alpha^j$;
		\item[$(ii)$] $\lda_\alpha(g^{n-i}x^i)=(-1)^i q^{\frac{i(i+1)}{2}} \alpha^i$;
		\item[$(iii)$] $\lda_\alpha(g^{n-1}x^j)=-q\ (j)_q\ \alpha^j$;
		\item[$(iv)$] $\lda_\alpha(g^{i}x^{n-1}) = \alpha^{n-1}$,
	\end{enumerate}
hold for all $i,j \in \I_{0,n-1}$.
\end{cor}

\begin{exas}
Note that all partial actions of $T_2(-1)= \mathbb{H}_4$ on its base field are already given in Example \ref{lda_sweedler}.
So,  here we exhibit the partial actions of $T_3(q)$ and $T_4(\omega)$ on their base fields.

Since $G(T_3(q)) = C_3$, the partial actions of $T_3(q)$ on $\k$ are given by $\varepsilon$ (global action) and $\lda_\alpha$, for any $\alpha \in \k$, where $\lda_\alpha(g^i)=\delta_{i,0}$, $\lda_\alpha(x)=\alpha$, $\lda_\alpha(gx)=0$, $\lda_\alpha(g^2x)= - q \alpha$ and $\lda_\alpha(g^ix^2)=\alpha^2$, for all $i \in \I_{0,2}$.

For $T_4(\omega)$, we have the global action $\varepsilon$ and the genuine partial actions $\lda^0_{\{1, g^2\}}$ and $\lda_\beta$, for any $\beta \in \k$.
The partial action $\lda^0_{\{1, g^2\}}$, associated to the unique non-trivial subgroup $\{1, g^2\}$ of $G(T_4(\omega)) = C_4$, is given by $\lda^0_{\{1, g^2\}}(g^ix^j) = \delta_{j,0}\lda_{\{1, g^2\}}(g^i)$, for all $i,j \in \I_{0,3}$ and the partial action $\lda_\beta$ is defined as $\lda_\beta(g^i)=\delta_{i,0}$, $\lda_\beta(x)=\beta$, $\lda_\beta(gx)=\lda_\beta(g^2x)=0$, $\lda_\beta(g^3x)= - \omega \beta$, $\lda_\beta(x^2)=\beta^2$, $\lda_\beta(gx^2)=0$, $\lda_\beta(g^2x^2)=- \omega \beta^2$, $\lda_\beta(g^3x^2)= (1- \omega)\beta^2$, and $\lda_\beta(g^ix^3)=\beta^3$, for all $i \in \I_{0,3}$.
\end{exas}

\medbreak

For a partial action to define a partial representation, it must be symmetric \cite[Example 3.5]{Partial_representations}.
Our next goal is to verify that all partial actions of $T_n(q)$ on its base field are symmetric.
Recall that a partial action $\lda: H \longrightarrow \k$ is symmetric if satisfies the additional condition $\lda(h)\lda(k) = \lda(h_1k)\lda(h_2)$, for all $h,k \in H$ (see Lemma \ref{k_mod_alg_parc}).
For this purpose, we need the next result.

\begin{lema}\label{lema4}
	For $i, j, t, s \in \N_0$ and $q \in \k^{\times}$, the following equality holds
	\begin{align*}
q^{s(i-j)} \sum_{\ell=0}^{j} {j \choose \ell}_q {{j+t-\ell} \choose {i+s-\ell}}_q {\ell \choose i}_q (-1)^{i-\ell} q^{\frac{(i-\ell)(i-\ell+1)}{2}}=	{j \choose i}_q {t \choose s}_q.
	\end{align*}
\end{lema}

\begin{proof}
	We prove by induction on $j$.
	First, note that for $j \in \N_0$ fixed, if $i>j$ the desired equality holds since ${j \choose i}_q = 0$ and also ${\ell \choose i}_q = 0$, for all $\ell \in \I_{0,j}$.
	Hence we can assume $i \leq j$.
	Using Lemma \ref{prodqbinom} is sufficient to prove that
	\begin{align*}
	q^{s(i-j)}\sum_{\ell=i}^{j} {{j-i} \choose {\ell-i}}_q {{j+t-\ell} \choose {i+s-\ell}}_q (-1)^{i-\ell} q^{\frac{(i-\ell)(i-\ell+1)}{2}} = {t \choose s}_q,
	\end{align*}
	for all $j, i, t, s \in \N_0$ with $i \in \I_{0,j}$.
	If $j=0$, then it follows that $i=0$ and the equality holds.
	For the induction step,
	\begin{align*}
	& q^{s(i-j)}\sum_{\ell=i}^{j} {{j-i} \choose {\ell-i}}_q {{j+t-\ell} \choose {i+s-\ell}}_q \, (-1)^{i-\ell} q^{\frac{(i-\ell)(i-\ell+1)}{2}} \\
	= & \, q^{s(i-j)} \left[ {{j+t-i} \choose s }_q \right. + \sum_{\ell=i+1}^{j-1} {{j-i} \choose {\ell-i}}_q {{j+t-\ell} \choose {i+s-\ell}}_q \\
	\times & \left. (-1)^{i-\ell} q^{\frac{(i-\ell)(i-\ell+1)}{2}} + { t \choose {i+s-j}}_q (-1)^{i-j} q^{\frac{(i-j)(i-j+1)}{2}} \right] \\
	\stackrel{ \eqref{6.5}}{=} & q^{s(i-j)} \left[ {{j+t-i} \choose s }_q \right. \\
	+ & \sum_{\ell=i+1}^{j-1} \left( {{j-i-1} \choose {\ell-i-1}}_q + q^{\ell-i}{{j-i-1} \choose {\ell-i}}_q \right) \\
	\times & {{j+t-\ell} \choose {i+s-\ell}}_q (-1)^{i-\ell} q^{\frac{(i-\ell)(i-\ell+1)}{2}} \\
	+ & \left. { t \choose {i+s-j}}_q (-1)^{i-j} q^{\frac{(i-j)(i-j+1)}{2}} \right] \\
	= & \, q^{s(i-j)} \left[ {{j+t-i} \choose s }_q \right. \\
	+ & \sum_{\ell=i+1}^{j-1} {{j-i-1} \choose {\ell-i-1}}_q {{j+t-\ell} \choose {i+s-\ell}}_q (-1)^{i-\ell} q^{\frac{(i-\ell)(i-\ell+1)}{2}} \\
	+ & \sum_{\ell=i+1}^{j-1} q^{\ell-i}{{j-i-1} \choose {\ell-i}}_q {{j+t-\ell} \choose {i+s-\ell}}_q (-1)^{i-\ell} q^{\frac{(i-\ell)(i-\ell+1)}{2}} \\
	+ & \left. { t \choose {i+s-j}}_q (-1)^{i-j} q^{\frac{(i-j)(i-j+1)}{2}} \right] \\
	= & \, q^{s(i-j)} \left[ {{j+t-i} \choose s }_q - {{j+t-i-1} \choose {s-1}}_q  \right. \\
	+ & \sum_{\ell=i+2}^{j-1} {{j-i-1} \choose {\ell-i-1}}_q {{j+t-\ell} \choose {i+s-\ell}}_q (-1)^{i-\ell} q^{\frac{(i-\ell)(i-\ell+1)}{2}} \\
	+ & \sum_{\ell=i+1}^{j-2} q^{\ell-i}{{j-i-1} \choose {\ell-i}}_q {{j+t-\ell} \choose {i+s-\ell}}_q (-1)^{i-\ell} q^{\frac{(i-\ell)(i-\ell+1)}{2}} \\
	+ & \, q^{j-1-i} {{t+1} \choose {i+s-j+1}}_q (-1)^{i-j+1} q^{\frac{(i-j+1)(i-j+2)}{2}}  \\
	+ & \left. { t \choose {i+s-j}}_q (-1)^{i-j} q^{\frac{(i-j)(i-j+1)}{2}} \right] \\
	= & \, q^{s(i-j)} \left[ \left( {{j+t-i} \choose s }_q - {{j+t-i-1} \choose {s-1}}_q  \right) \right. \\
	+ & \sum_{\ell=i+1}^{j-2} {{j-i-1} \choose {\ell-i}}_q {{j+t-\ell-1} \choose {i+s-\ell-1}}_q (-1)^{i-\ell-1} q^{\frac{(i-\ell-1)(i-\ell)}{2}} \\
	+ & \sum_{\ell=i+1}^{j-2} {{j-i-1} \choose {\ell-i}}_q {{j+t-\ell} \choose {i+s-\ell}}_q (-1)^{i-\ell} q^{\frac{(i-\ell-1)(i-\ell)}{2}} \\
	+ & \left. \left( {{t+1} \choose {i+s-j+1}}_q - { t \choose {i+s-j}}_q  \right) (-1)^{i-j+1} q^{\frac{(i-j)(i-j+1)}{2}} \right] \\
	\stackrel{ \eqref{6.5}}{=}& \, q^{s(i-j)} \left[ q^s {{j+t-i-1} \choose s }_q  \right. \\
	+ & \sum_{\ell=i+1}^{j-2} {{j-i-1} \choose {\ell-i}}_q \left( {{j+t-\ell} \choose {i+s-\ell}}_q - {{j+t-\ell-1} \choose {i+s-\ell-1}}_q \right) \\
	\times & (-1)^{i-\ell} q^{\frac{(i-\ell-1)(i-\ell)}{2}}  \\
	+ & \left. q^{i+s-j+1}{{t} \choose {i+s-j+1}}_q  (-1)^{i-j+1} q^{\frac{(i-j)(i-j+1)}{2}} \right] \\
	\stackrel{ \eqref{6.5}}{=} & \, q^{s(i-j)} \left[ q^s {{j+t-i-1} \choose s }_q  \right. \\
	+ & \sum_{\ell=i+1}^{j-2} {{j-i-1} \choose {\ell-i}}_q q^{i+s-\ell}{{j+t-\ell-1} \choose {i+s-\ell}}_q (-1)^{i-\ell} q^{\frac{(i-\ell-1)(i-\ell)}{2}} \\
	+ & \left. \, q^s {{t} \choose {i+s-j+1}}_q  (-1)^{i-j+1}  q^{\frac{(i-j+1)(i-j+2)}{2}} \right] \\
	= & \, q^{s(i-j)} q^s \left[ {{j+t-i-1} \choose s }_q  \right. \\
	+ & \sum_{\ell=i+1}^{j-2} {{j-i-1} \choose {\ell-i}}_q {{j+t-\ell-1} \choose {i+s-\ell}}_q (-1)^{i-\ell} q^{\frac{(i-\ell)(i-\ell+1)}{2}} \\
	+ & \left. {{t} \choose {i+s-j+1}}_q  (-1)^{i-j+1} q^{\frac{(i-j+1)(i-j+2)}{2}} \right] \\
	= & \, q^{s(i-j+1)} \left[ \sum_{\ell=i}^{j-1} {{j-i-1} \choose {\ell-i}}_q {{j+t-\ell-1} \choose {i+s-\ell}}_q (-1)^{i-\ell} q^{\frac{(i-\ell)(i-\ell+1)}{2}} \right]. \qedhere
	\end{align*}
\end{proof}

\begin{teo}\label{simetria_taft}
	Every partial action of the Taft algebra $T_n(q)$ on $\k$ is symmetric.
\end{teo}
\begin{proof}
We need to check condition \eqref{eqn_lda_sim} for the partial actions $\lda_N^0$ and $\lda_\alpha$.
	Considering $\lda_\alpha$, $\alpha \in \k$,
		\begin{align*}
	\lda_\alpha(g^{n-i}x^{j}) \lda_\alpha(g^{n-s}x^{t}) = & \, q^{\frac{i(i+1)}{2}} {j \choose i}_q  (-1)^i \alpha^j q^{\frac{s(s+1)}{2}} {t \choose s}_q  (-1)^s \alpha^t \\
	= &  \, (-1)^{i+s} \alpha^{j+t} q^{\frac{i(i+1)}{2}+\frac{s(s+1)}{2}} {j \choose i}_q {t \choose s}_q,
	\end{align*}
	for all $i,j,s,t \in \I_{0,n-1}$. On the other hand,
		\begin{align*}
& \lda_\alpha((g^{n-i}x^{j})_1g^{n-s}x^{t}) \lda_\alpha((g^{n-i}x^{j})_2) \\
= & \sum_{\ell=0}^{j} { j \choose \ell }_q \lda_\alpha(g^{n-i+\ell}x^{j-\ell} g^{n-s}x^{t}) \lda_\alpha(g^{n-i}x^{\ell}) \\ 
= & \sum_{\ell=0}^{j} { j \choose \ell }_q q^{(j-\ell)(n-s)}\lda_\alpha(g^{n-(i-\ell+s)}x^{j+t-\ell}) \lda_\alpha(g^{n-i}x^{\ell}) \\
= & \sum_{\ell=0}^{j} { j \choose \ell }_q q^{s(\ell-j)}q^{\frac{(i+s-\ell)(i+s-\ell+1)}{2}} { {j+t-\ell} \choose {i+s-\ell} }_q  (-1)^{i+s-\ell} \alpha^{j+t-\ell} \lda_\alpha(g^{n-i}x^{\ell}) \\
= & \ q^{\frac{i(i+1)}{2}} \alpha^{j+t} \sum_{\ell=0}^{j} { j \choose \ell }_q { {j+t-\ell} \choose {i+s-\ell} }_q  {\ell \choose i}_q  (-1)^{s-\ell}  q^{s(\ell-j)} q^{\frac{(i+s-\ell)(i+s-\ell+1)}{2} }\\
= & \ q^{\frac{i(i+1)}{2}}  \alpha^{j+t} \sum_{\ell=0}^{j} { j \choose \ell }_q { {j+t-\ell} \choose {i+s-\ell} }_q  {\ell \choose i}_q  (-1)^{s-\ell}  q^{\frac{s(s+1)}{2}} q^{s(i-j)} q^{\frac{(i-\ell)(i-\ell+1)}{2}} \\
= & \ q^{\frac{i(i+1)}{2}} q^{\frac{s(s+1)}{2}} q^{s(i-j)} \alpha^{j+t} \sum_{\ell=0}^{j} { j \choose \ell }_q { {j+t-\ell} \choose {i+s-\ell} }_q  {\ell \choose i}_q  (-1)^{s-\ell}  q^{\frac{(i-\ell)(i-\ell+1)}{2}}.
	\end{align*}
	Thus, it follows that $\lda_\alpha$ is a symmetric partial action by Lemma \ref{lema4}.
	
	Now, consider the partial action $\lda_N^0$.
	Since $\lda_{C_n}^0 = \varepsilon$ and $\lda_{\{1\}}^0 = \lda_0$, we can assume $\{1\} \subsetneq N \subsetneq G(T_n(q))$.
	For all $i,j,s,t \in \I_{0,n-1}$,
		\begin{align*}
	\lda_N^0(g^{i}x^{j}) \lda_N^0(g^{s}x^{t}) =  \delta_{j,0}\delta_{t,0}\lda_N(g^i)\lda_N(g^s).
	\end{align*}
	On the other hand,
		\begin{align*}
	& \lda_N^0((g^{i}x^{j})_1 g^{s}x^{t})\lda_N^0((g^{i}x^{j})_2) \\
	= &  \sum_{\ell=0}^{j} { j \choose \ell }_q \lda_N^0 (g^{i+\ell}x^{j-\ell} g^{s}x^{t}) \lda_N^0(g^{i}x^{\ell}) \\
	= & \sum_{\ell=0}^{j} { j \choose \ell }_q q^{s(j-\ell)} \lda_N^0(g^{i+s+\ell}x^{j+t-\ell}) \lda_N^0(g^{i}x^{\ell}) \\
	= & \sum_{\ell=0}^{j} { j \choose \ell }_q q^{s(j-\ell)} \delta_{j+t-\ell,0}\lda_N(g^{i+s+\ell}) \delta_{\ell,0}\lda_N(g^{i}) \\
	= & \ q^{sj} \delta_{j+t,0}\lda_N(g^{i+s})\lda_N(g^{i}).
	\end{align*}
	Since $g^i$ is a group-like element, 
	\begin{align*}
	\lda_N(g^{i})\lda_N(g^s) \stackrel{ \eqref{eq}}{=}  \lda_N((g^{i})_1)\lda_N((g^{i})_2g^s)= 
	\lda_N(g^{i}) \lda_N(g^{i+s}).
	\end{align*}
	Moreover, for $j,t \in \I_{0,n-1}$, it follows that $j+t=0$ if and only if $j=t=0$, and then we obtain $\delta_{j,0}\delta_{t,0} = q^{sj} \delta_{j+t,0}$.
	Therefore, $\lda_N^0$ is also a symmetric partial action of $T_n(q)$ on $\k$.
\end{proof}

	From Lemma \ref{lema4}, we also obtain an interesting $q$-identity:
	\begin{align*}
	q^{-sj} \sum_{\ell=0}^{j} {j \choose \ell}_q {{j+t-\ell} \choose {s-\ell}}_q (-1)^{-\ell} q^{\frac{\ell(\ell-1)}{2}} = {t \choose s}_q,
	\end{align*}
for all $j, t, s \in \N_0$ and $q \in \k^{\times}$. 

\subsection{Partial Coactions}\label{Subsec:coacoes_taft}
In this subsection, since Taft algebras are self-dual Hopf algebras, all partial coactions of $T_n(q)$ on $\k$ are calculated from the partial actions.

\medbreak 

By Proposition \ref{dual_corpo}, the linear map $\lda: T_n(q) \longrightarrow \k $ defines a partial action of $T_n(q)$ on $\k$ if and only if the linear map $\rho_\lda: \k \longrightarrow \k \o (T_n(q))^*$, given by $\rho_\lda(1_\k) = 1_\k \o \lda $, defines a partial coaction of $(T_n(q))^*$ on $\k$.

Since all partial actions of $T_n(q)$ on $\k$ were classified in Theorem \ref{taft} and $T_n(q)$ is a self-dual Hopf algebra, all partial coactions of $T_n(q)$ on $\k$ are given by $\rho_{\varphi(\lda)}$, where $\varphi: (T_n(q))^* \longrightarrow T_n(q)$ is an isomorphism of Hopf algebras and $\lda$ is as in \eqref{def_lda_alpha} and \eqref{def_lda_N_zero}.
Moreover, Theorem \ref{simetria_taft} and Proposition \ref{dual_corpo} ensure that all partial coactions of $T_n(q)$ on $\k$ are symmetric.

However, we are interested in an explicit presentation for the idempotent elements in $T_n(q)$ that define the partial coactions of $T_n(q)$ on $\k$.
For this purpose, we present explicitly an isomorphism of Hopf algebras $\varphi : (T_n(q))^* \longrightarrow T_n(q)$.

\begin{lem}\label{inversa}
	Consider the Taft algebra $T_n(q)$ and the linear maps:
	\begin{itemize}
		\item[(1)] $\psi : T_n(q) \longrightarrow (T_n(q))^*$ given by
		$\psi(g^ix^j)=G^iX^j,$ where 
		\begin{align*}
		G^iX^j  = \sum_{k=0}^{n-1} (j)_q! q^{-i(k+j)-jk-\frac{j(j-1)}{2}} (g^kx^j)^*;
		\end{align*}
		
		\item[(2)] $\varphi : (T_n(q))^* \longrightarrow T_n(q)$ given by
		\begin{align*}
		\varphi ((g^ix^j)^*)=\frac{1}{n}((j)_q!)^{-1} q^{ij+\frac{j(j-1)}{2}} \sum_{k=0}^{n-1} q^{k(i+j)} g^kx^j. 
		\end{align*}
	\end{itemize}
Then, both maps defined in (1) and (2) are isomorphisms of Hopf algebras.
Furthermore, $\varphi = \psi^{-1}$.
\end{lem}

The next lemma is useful to exhibit explicitly the idempotent element  $z=\varphi(\lda_N^0) \in T_n(q)$, which corresponds to the partial action of the Taft algebra $T_n(q)$ on $\k$ associated with a non-trivial subgroup $N$ of $C_n$, when it exists.

\begin{lema}\label{zdogrupo}
	Let $n \geq 2$ be an integer, $ t \in \I_{0,n-1}$, $q$ a primitive $n^{th}$ root of unity and $ C_n = \{ 1, g, g^2 , ..., g^{n-1}\} $.
	Then, 
	$$ \frac{1}{n} \sum_{t=0}^{n-1}  \left( \sum_{i=0}^{\ell-1}  q^{ikt} \right)  g^t = \frac{\ell}{n}\sum_{i=0}^{k-1}g^{i \ell},$$
	for all $k,\ell$ positive integers such that $n=k \ell$.
\end{lema}
\begin{proof}
	First, note that $t \in \I_{0,n-1}$ may be uniquely written as $t=u+s\ell$, for $ u \in \I_{0,\ell-1}$ and $ s \in \I_{0,k-1}$.
	Then
	\begin{align*}
	\frac{1}{n} \sum_{t=0}^{n-1}  \left( \sum_{i=0}^{\ell-1}  q^{ikt} \right)  g^t & = \frac{1}{n} \sum_{s=0}^{k-1}\sum_{u=0}^{\ell-1}  \left( \sum_{i=0}^{\ell-1}  q^{ik(u+s\ell)} \right)  g^{u+s\ell} \\
	& = \frac{1}{n} \sum_{s=0}^{k-1} \left(  \left( \sum_{i=0}^{\ell-1}  q^{iks\ell}   \right) g^{s\ell} + \sum_{u=1}^{\ell-1} \left( \sum_{i=0}^{\ell-1}  q^{ik(u+s\ell)}  \right)  g^{u+s\ell}  \right).
	\end{align*}
	Since $q^{iks \ell} = q^{i s n} = 1$ and $q^{ku} \neq 1$ is a $\ell^{th}$ root of unity, it follows that
	\begin{align*}
	\sum_{i=0}^{\ell-1} q^{ik(u+s \ell)}  = \sum_{i=0}^{\ell-1} q^{iku + i s n} =\sum_{i=0}^{\ell-1} (q^{ku})^i =  0,
	\end{align*}
	for all $u \in \I_{\ell-1}$.
\end{proof}

\medbreak 

To state the partial coactions' theorem of $T_n(q)$ on $\k$, we define:
\begin{enumerate}
	\item for each $\alpha \in \k$, the idempotent element $z_\alpha \in T_n(q)$ as
	\begin{equation}\label{def_z_alpha}
z_\alpha = \frac{1}{n} \sum_{k=0}^{n-1}  g^k  + \frac{1}{n} \sum_{k=0}^{n-1} \sum_{j=1}^{n-1} q^{\frac{j(j-1)}{2}+kj} \alpha^j  \left( \sum_{i=0}^{j}  \frac{(-1)^i  q^{\frac{i(i+1)}{2}-i(j+k)} }{(j-i)_q! \, (i)_q!}  \right) g^kx^j.
	\end{equation}
	Note that $z_0 = \frac{1}{n} \sum_{k=0}^{n-1}  g^k   \in G(T_n(q))$.
	
	\smallbreak
	
	\item for each subgroup $N$ of $C_n = G(T_n(q))$, the idempotent element $z_N \in T_n(q)$ as
	\begin{equation}\label{def_z_N}
z_N = \frac{1}{|N|}\sum_{h \in N} h.
	\end{equation}
	In particular, $z_{\{1\}} =1$ and $z_{C_n} = z_0$  (as in \eqref{def_z_alpha}).
\end{enumerate}

\begin{teo}
	An element $z \in T_n(q)$ is a partial coaction of $T_n(q)$ on $\k$ if and only if $z=1$ (global coaction), $z = z_N$ or $z = z_\alpha$, where $\alpha \in \k$ and $N$ is a non-trivial subgroup of $G (T_n(q))$.
	Moreover, all such partial coactions are symmetric.
\end{teo}
\begin{proof}
	The considerations at the beginning of this subsection ensure that the partial coactions of $T_n(q)$ on $\k$ are given by $ \varphi(\lda)$, for $\lda$ a partial action of $T_n(q)$ on $\k$ as in Theorem \ref{taft} and $\varphi$ as in Lemma \ref{inversa}, and all such partial coactions are symmetric.
	
	From Theorem \ref{taft}, all partial actions of $T_n(q)$ on $\k$ are $\varepsilon, \lda_\alpha$ and $\lda_N^0$, where $\alpha \in \k$ and $N$ is a non-trivial subgroup of $G(T_n(q))$.
	Then, we will explicitly provide the idempotent elements $\varphi(\varepsilon), \varphi(\lda_\alpha), \varphi(\lda_N^0) \in T_n(q)$.
		
	Since $\varphi$ is an isomorphism of Hopf algebras, clearly $\varphi(\varepsilon)=1$.
	Now, suppose there exists a non-trivial subgroup  of $G(T_n(q))$, \emph{i.e.}, there exist positive integers $k, \ell \in \I_{2, n-1}$ such that $k \ell = n$.
	Consider the non-trivial subgroup $\langle g^k \rangle$ of $G(T_n(q))$ and its corresponding partial action
	 $\lda_{\langle g^k \rangle}^0 =\sum_{i=0}^{\ell-1} (g^{ik})^*$ (see \eqref{lda_N_base_dual}).
	 Then, 
	\begin{align*}
	\varphi\left(\lda_{\langle g^k \rangle}^0 \right) = & \ \varphi \left(\sum_{i=0}^{\ell-1} (g^{ik})^* \right) =  \sum_{i=0}^{\ell-1} \varphi((g^{ik})^*) 
	=  \sum_{i=0}^{\ell-1} \left( \frac{1}{n} \sum_{t=0}^{n-1} q^{ikt} g^t\right) \\ 
	= & \ \frac{1}{n} \sum_{i=0}^{\ell-1} \sum_{t=0}^{n-1} q^{ikt} g^t 
	=  \frac{1}{n} \sum_{t=0}^{n-1}  \left( \sum_{i=0}^{\ell-1}  q^{ikt} \right)  g^t 
	= \frac{\ell}{n}\sum_{i=0}^{k-1}g^{i\ell}, 
	\end{align*}
	where the last equality follows from Lemma \ref{zdogrupo}.
	Thus, $\varphi\left(\lda_{\langle g^k \rangle}^0 \right) = z_{\langle g^\ell \rangle}$.
	
	Now, for $\lda_{\alpha} = \sum_{0 \leq i \leq j \leq n-1} q^{\frac{i(i+1)}{2}} {j \choose i}_q  (-1)^i \alpha^j (g^{n-i}x^j)^*$, where $\alpha \in \k$,	 
	\begin{align*}
	\varphi ( \lda_\alpha)  = & \ \varphi \left( \sum_{0 \leq i \leq j \leq n-1} q^{\frac{i(i+1)}{2}} {j \choose i}_q  (-1)^i \alpha^j (g^{n-i}x^j)^* \right) \\
	= & \ \sum_{0 \leq i \leq j \leq n-1} q^{\frac{i(i+1)}{2}} {j \choose i}_q  (-1)^i \alpha^j \varphi( (g^{n-i}x^j)^* ) \\
	= & \ \sum_{0 \leq i \leq j \leq n-1} q^{\frac{i(i+1)}{2}} {j \choose i}_q  (-1)^i \alpha^j \left( \frac{ q^{(n-i)j+\frac{j(j-1)}{2}} }{n(j)_q!} \sum_{k=0}^{n-1} q^{k(n-i+j)} g^kx^j \right) \\
	= & \ \frac{1}{n} \sum_{j=0}^{n-1} \sum_{i=0}^{j}  \frac{(-1)^i \alpha^j q^{\frac{i(i+1)}{2}+\frac{j(j-1)}{2}-ij} }{(j-i)_q! \, (i)_q!} \left( \sum_{k=0}^{n-1} (q^{j-i})^k g^kx^j \right) \\
	= & \ \frac{1}{n} \sum_{k=0}^{n-1} \sum_{j=0}^{n-1} q^{\frac{j(j-1)}{2}+kj} \alpha^j \left( \sum_{i=0}^{j}  \frac{(-1)^i  q^{\frac{i(i+1)}{2}-i(j+k)} }{(j-i)_q! \, (i)_q!} \right) g^kx^j\\
	= & \ \frac{1}{n} \sum_{k=0}^{n-1} \left( g^k  + \sum_{j=1}^{n-1} q^{\frac{j(j-1)}{2}+kj} \alpha^j  \left( \sum_{i=0}^{j}  \frac{(-1)^i  q^{\frac{i(i+1)}{2}-i(j+k)} }{(j-i)_q! \, (i)_q!}  \right) g^kx^j \right) \\
	= & \ \frac{1}{n} \sum_{k=0}^{n-1}  g^k  + \frac{1}{n} \sum_{k=0}^{n-1} \sum_{j=1}^{n-1} q^{\frac{j(j-1)}{2}+kj} \alpha^j  \left( \sum_{i=0}^{j}  \frac{(-1)^i  q^{\frac{i(i+1)}{2}-i(j+k)} }{(j-i)_q! \, (i)_q!}  \right) g^kx^j.
	\end{align*}
	Therefore, $\varphi ( \lda_\alpha) = z_\alpha$ as defined in \eqref{def_z_alpha}.
\end{proof}

\begin{exas}
For Sweedler's $4$-dimensional Hopf algebra, $T_2(-1) = \mathbb{H}_{4}$, we have only the partial coactions $\rho_z$ given by $z=1$, which is the global one, and $z_\alpha = \frac{1 + g}{2} - \alpha gx,$
where $\alpha \in \k$.
Recall that such partial coactions were presented in 
Example \ref{z_sweedler}.

Now we present the partial coactions of $T_3(q)$ and $T_4(\omega)$ on their base fields, where $q$ and $\omega$ are primitive roots of unity of order $3$ and $4$, respectively.

Since $G(T_3(q)) = C_3$ has only the trivial subgroups, the partial coactions of $T_3(q)$ on $\k$ are $z=1$ (global coaction) and $z_\alpha$, for any $\alpha \in \k$, where
	$$z_\alpha = \frac{1 + g + g^2 }{3} + \frac{1}{3}\left( (q-1)\alpha gx + (q^2-1) \alpha g^2x - 3q \alpha^2gx^2 \right).$$
	
For $T_4(\omega)$, we have the global coaction $z=1$, the partial coaction $z_{\{1,g^2\}} = \frac{1+g^2}{2}$ associated to the unique non-trivial subgroup $\{1, g^2\}$ of $C_4$, and $z_\beta$, where
	\begin{align*}
	z_\beta  = & \ \frac{1}{4}\left( \right. 1 + g + g^2 + g^3 + \omega(1+\omega)\beta gx - 2 \beta g^2x - (1+\omega) \beta g^3x \\
	- & \ \left.  2 \omega \beta^2 gx^2  + 2 \beta^2 g^2x^2 + 2(1+\omega) \beta^3gx^3 \right),
	\end{align*}
	for any $\beta \in \k$.
\end{exas}

\section{Partial (Co)Actions of the Nichols Hopf Algebra $\mathbb{H}_{2^n}$ on $\k$}\label{sec:Nichols}

In this section, we reproduce ideas and techniques from previous section to calculate all partial (co)actions of the Nichols Hopf algebra $\mathbb{H}_{2^n}$ on its base field $\k$.

\medbreak

We emphasize that the so-called Nichols Hopf algebras were introduced by Taft in \cite{Taft}.
But, this family of Hopf algebras was named after the work of Nichols \cite{nichols}.
Such Hopf algebras are the prototype for the theory of \emph{Nichols algebras}.
We will present this family of Hopf algebras as it appears in \cite[Subsection 2.2]{etingof}.

\medbreak

Let $n \geq 2$ be an integer and suppose $char(\k) \neq 2$.
The \emph{Nichols Hopf algebra of order $n$}, or shortly \emph{Nichols Hopf algebra}, here denoted by $\mathbb{H}_{2^n}$, has the structure as follows:
as algebra it is generated over $\k$ by the $n$ letters $g,x_1, \cdots, x_{n-1}$ with relations $g^2=1$, $x_i^2=0$, $x_i g = -g x_i$ and $x_ix_j = -x_jx_i$, for all $i, j \in \I_{n-1}$. 
Thus, the set
 $\mathcal{B} = \{g^{j_0}x_1^{j_1}x_2^{j_2}...x_{n-1}^{ j_{n-1} }  : j_i \in \I_{0,1}, i \in \I_{0, n-1} \}$
 is the canonical basis for $\mathbb{H}_{2^n}$ and consequently $dim_\k(\mathbb{H}_{2^n}) = 2^{n}$.
To complete the Hopf algebra structure of $\mathbb{H}_{2^n}$, we set
$\Delta (g) = g \o g$, $\varepsilon (g) = 1$, $S(g) =g^{-1}=g,$ and $\Delta (x_i) = x_i \o 1 + g \o x_i$, $\varepsilon (x_i) = 0$ and $S(x_i) = -gx_i,$ for all $i \in \I_{n-1}$. 
Note that  $G(\mathbb{H}_{2^n})= \{1, g \} = C_2$, for any $n$.
In particular, when $n=2$, the Nichols Hopf algebra $\mathbb{H}_{2^2}$ is exactly the Sweedler's $4$-dimensional Hopf algebra $\mathbb{H}_{4}$.

Nichols Hopf algebra also is a self-dual Hopf algebra, and again this property allow us to determine all partial coactions from partial actions on its base field.

\subsection{Partial Actions}
In this subsection all partial actions of $\mathbb{H}_{2^n}$ on $\k$ are classified.

\medbreak

First, we present a family of genuine partial actions of $\mathbb{H}_{2^n}$ on $\k$.
Such a family is parameterized by $\k^{n-1}$, as follows.

\begin{prop}\label{lda_alpha_eh_acao_nichols}
	For any  $\alpha=(\alpha_i)_{i \in \I_{n-1}} \in \k^{n-1}$, the linear map $\lda_\alpha \in (\mathbb{H}_{2^n})^*$,
	defined as 	$\lda_\alpha = 1^* + \sum_{i=1}^{n-1}\alpha_i \, [ (x_i)^* + (gx_i)^*]$, is a partial action.
\end{prop}
\begin{proof}
	We need to check if condition \eqref{eq} holds.
	We proceed fixing $h \in \mathcal{B}$ and verifying \eqref{eq} for any $y \in \mathcal{B}$, where $\mathcal{B}$ is the canonical basis of $\mathbb{H}_{2^n}$.
	
	First, if $h \in \{1, g, x_1, \cdots, x_{n-1} \}$, then \eqref{eq} holds trivially for any $y \in \mathcal{B}$.
	
	If $h=gx_i$, $i \in I_{n-1}$, then \eqref{eq} means 
	$\lda_\alpha(gx_i)\lda_\alpha(y) =\lda_\alpha(gx_i) \lda_\alpha(g y) + \lda_\alpha(gx_i y),$
	for any $y \in \mathcal{B}$.
	It is a routine calculation to verify that previous equality holds for any $y \in \mathcal{B}$.
	One can do it considering $y$ as $g^\ell, g^\ell x_j$ and $g^\ell x_{j_1}x_{j_2}...x_{j_s}$, where $\ell \in \I_{0,1},$ $j \in \I_{n-1}$ and $s \in \I_{2,n-1}, j_t \in \I_{n-1}, t \in \I_s$.
	
	Now, if $h=g^\ell x_{i_1}x_{i_2}$, $\ell \in \I_{0, 1}$, $i_1, i_2 \in \I_{n-1}$, $i_1 < i_2$, then
	\begin{align*}
	\Delta(g^\ell x_{i_1}x_{i_2}) = & \ g^\ell x_{i_1}x_{i_2} \o g^\ell - g^{\ell+1}x_{i_1} \o g^\ell x_{i_2} + g^{\ell +1} x_{i_2} \o g^{\ell}x_{i_1} \\
	+ &\  g^\ell \o g^\ell x_{i_1}x_{i_2},
	\end{align*}
	and so \eqref{eq} becomes
	\begin{align*}
	\lda_\alpha(g^\ell x_{i_1}x_{i_2})\lda_\alpha(y) = & \ \lda_\alpha(g^\ell x_{i_1}x_{i_2})  \lda_\alpha(g^\ell y) - \lda_\alpha(g^{\ell+1}x_{i_1}) \lda_\alpha(g^\ell x_{i_2} y) \\
	+ & \ \lda_\alpha(g^{\ell+1} x_{i_2}) \lda_\alpha(g^{\ell}x_{i_1} y) +  \lda_\alpha(g^\ell) \lda_\alpha( g^\ell x_{i_1}x_{i_2} y),
	\end{align*}
	for any $y \in \mathcal{B}$.

	Since $\lda_\alpha(gx_s)=\lda_\alpha(x_s)=\alpha_s$ and $\lda_\alpha(g^\ell x_{i_1}x_{i_2})=\lda_\alpha(g^\ell x_{i_1}x_{i_2}y)=0$, for any $s \in \I_{n-1}$, $y \in \mathcal{B}$, \eqref{eq} is reduced to
	$$\alpha_{i_1} \lda_\alpha(g^\ell x_{i_2}y) = \alpha_{i_2} \lda_\alpha(g^\ell x_{i_1}y),$$
	for any $ y \in \mathcal{B}$.
	One can verify this equality considering $y$ as $g^t$ and $g^t x_{j_1}...x_{j_s}$, where $t \in \I_{0,1}$ and $s \in \I_{n-1}$, $j_k \in \I_{n-1}$, $k \in \I_{s}$, $j_1 < \cdots < j_s$.
	
	Finally, it remains to check \eqref{eq} for $h = g^\ell x_{i_1}x_{i_2}...x_{i_s}$, where $\ell \in \I_{0,1}$ and $s \geq 3$, and any $y \in \mathcal{B}$. 
	Write $h =g^\ell x_{i_1}x_{i_2}x_{i_3}w$, for some $w \in \mathbb{H}_{2^n}$.
	Then,
	\begin{align*}
	\Delta(g^\ell  x_{i_1}x_{i_2}x_{i_3}w) & =  \Delta(g^\ell  x_{i_1}x_{i_2})\Delta(x_{i_3}) \Delta(w) \\
	& =  g^\ell x_{i_1}x_{i_2}x_{i_3}w_1 \o g^\ell w_2 - g^{\ell +1}x_{i_1}x_{i_3}w_1 \o g^\ell x_{i_2}w_2 \\
	& + g^{\ell +1} x_{i_2}x_{i_3}w_1 \o g^{\ell }x_{i_1}w_2 + g^\ell  x_{i_3}w_1 \o g^\ell x_{i_1}x_{i_2}w_2 \\
	& + g^{\ell +1}x_{i_1}x_{i_2}w_1 \o g^\ell x_{i_3} w_2  + g^{\ell +2}x_{i_1}w_1 \o g^\ell x_{i_2}x_{i_3} w_2 \\
	& - g^{\ell +2} x_{i_2}w_1 \o g^{\ell }x_{i_1}x_{i_3} w_2 + g^{\ell +1} w_1 \o g^\ell x_{i_1}x_{i_2}x_{i_3} w_2,
	\end{align*}
	and for any $y \in \mathcal{B}$, condition \eqref{eq} means
	\begin{align*}
	\lda_\alpha(g^\ell  x_{i_1}x_{i_2}x_{i_3}w)\lda_\alpha(y) & = \lda_\alpha( g^\ell x_{i_1}x_{i_2}x_{i_3}w_1)  \lda_\alpha(g^\ell w_2 y) \\
	& -  \lda_\alpha( g^{\ell +1}x_{i_1}x_{i_3}w_1)  \lda_\alpha(g^\ell x_{i_2}w_2 y)\\
	& +  \lda_\alpha( g^{\ell +1} x_{i_2}x_{i_3}w_1)  \lda_\alpha( g^{\ell }x_{i_1}w_2 y) \\
	& + \lda_\alpha( g^\ell  x_{i_3}w_1)  \lda_\alpha( g^\ell x_{i_1}x_{i_2}w_2 y) \\
	& +  \lda_\alpha( g^{\ell +1}x_{i_1}x_{i_2}w_1 )  \lda_\alpha( g^\ell x_{i_3} w_2 y) \\
	& + \lda_\alpha( g^{\ell +2}x_{i_1}w_1) \lda_\alpha( g^\ell x_{i_2}x_{i_3} w_2 y) \\
	& - \lda_\alpha( g^{\ell +2} x_{i_2}w_1 ) \lda_\alpha( g^{\ell }x_{i_1}x_{i_3} w_2 y) \\
	& + \lda_\alpha( g^{\ell +1} w_1 ) \lda_\alpha( g^\ell x_{i_1}x_{i_2}x_{i_3} w_2 y).
	\end{align*}
	
	Note that the above equality holds since $\lda_\alpha$ evaluated in a product that contains at least two skew-primitive elements as factors is equal to zero.
	
	Therefore,  the map $\lda_\alpha$ as given is a partial action of $\mathbb{H}_{2^n}$ on $\k$.
\end{proof}

Now we classify all partial actions of $\mathbb{H}_{2^n}$ on $\k$.

\begin{teo}\label{nichols}
	Let $\lda:\mathbb{H}_{2^n} \longrightarrow \k$ be a linear map.
	Then, $\lda$ is a partial action of $\mathbb{H}_{2^n}$ on $\k$ if and only if $\lda = \varepsilon$ (global action) or $\lda=\lda_\alpha$, where	
	$\alpha=(\alpha_i)_{i\in \I_{n-1}} \in \k^{n-1}$.
\end{teo}
\begin{proof}
	
Let $ \lda \in (\mathbb{H}_{2^n})^*$ be a partial action of $\mathbb{H}_{2^n}$ on $\k$.
We shall see that $\lda = \varepsilon$ or $\lda = \lda_\alpha$, for some $\alpha \in \k^{n-1}$, as defined in Proposition \ref{lda_alpha_eh_acao_nichols}.

Since $\lda |_{\k G(\mathbb{H}_{2^n})} = \lda_N$, where
$N =\{h \in G(\mathbb{H}_{2^n}) \, : \, \lda(h)=1_\k \}$ is a subgroup of $G(\mathbb{H}_{2^n})$, and 
 $G(\mathbb{H}_{2^n}) = \{1,g\}$ for any $n$, we have only two possibilities: $N = \{1,g\}$ or $N = \{1\}$.
	
	\smallbreak
	
	\underline{\textbf{$1^{st}$ case: $N=\{1, g\}$}}
		
	Consider $\lda(g)=\lda(1)=1$.
	First, since each $x_i$ is an $(1,g)$-primitive element, by Proposition \ref{propriedades_lda} (ii) it follows that $\lda(x_i)=0$ for all $i \in \I_{n-1}.$ 
	Then, by item (iii) of Proposition \ref{propriedades_lda}, $\lda(x_{j_1}x_{j_2}\cdots x_{j_\ell})=0$ for any $ \ell \in \I_{n-1}$ and $j_s \in \I_{n-1}$, $s \in \I_{\ell}$.
	
	Finally, by Proposition \ref{propriedades_lda} (i), we conclude $\lda(gx_{j_1}x_{j_2}\cdots x_{j_\ell})=0,$ for any $\ell, j_s \in \I_{n-1}$, $s \in \I_{\ell}$.
	
	Therefore, in this case, $\lda=\varepsilon$.
	
	\smallbreak
	
	\underline{\textbf{$2^{nd}$ case: $N=\{1\}$}}
	
	Here $\lda(g)=0$ by assumption.
	In order to determine the linear map $\lda$, we use condition \eqref{eq}.
	First, considering $h=gx_i$ and $y=g$ in \eqref{eq}, we obtain $\lda(gx_i)=\lda(x_i)$, for all $i \in \I_{n-1}$.
	Then, for $h=x_ix_j$ and $k=1$ we get $\lda(x_ix_j)=0$  and finally taking $h=x_ix_j$ and $k=g$ we deduce that $\lda(gx_ix_j)=0$, for all $i,j \in \I_{n-1}$, $i \neq j$.
	
	By induction, one can prove that $\lda(x_{i_1}x_{i_2}...x_{i_s})=\lda(gx_{i_1}x_{i_2}...x_{i_s})=0$,  for $s \geq 2$, $i_\ell \in \I_{n-1}$, $\ell \in \I_{s}$.
	Indeed, we already done for $s=2$.
	Assume $s \geq 2$ and consider $h=gx_{i_1}$, $k=x_{i_2}...x_{i_s}$ (resp. $h=gx_{i_1}$, $k=gx_{i_2}...x_{i_s}$) in \eqref{eq}.
	Then, using the induction hypothesis, it follows that $\lda(gx_{i_1}x_{i_2}...x_{i_s})=0$ (resp. $\lda(x_{i_1}x_{i_2}...x_{i_s})=0$).

	Thus, we have established the evaluation of the linear map $\lda$, namely $\lda(1)=1$, $\lda(g)=0$, $\lda(x_i)= \lda(g x_i)$ for each $i \in \I_{n-1}$, and $\lda(x_{i_1} x_{i_2} \cdots x_{i_s})=0$, $\lda(g x_{i_1} x_{i_2} \cdots x_{i_s})=0$, for all $s \in \I_{2,n-1}$, $i_j, i_\ell \in \I_{n-1}$, $j,\ell \in \I_{s}$.
	
	In this case, $\lda = \lda_\alpha$ as in Proposition \ref{lda_alpha_eh_acao_nichols}, where the parameter is $\alpha = (\lda(x_i))_{i \in \I_{n-1}} \in \k^{n-1}$.
\end{proof}

\begin{teo}\label{simetria_nichols}
	Every partial action of the Nichols Hopf algebra $\mathbb{H}_{2^n}$ on $\k$ is symmetric.
\end{teo}
\begin{proof}
	Consider $\alpha \in \k^{n-1}$ and the partial action $\lda_{\alpha}$ of $\mathbb{H}_{2^n}$ on $\k$ defined in Proposition \ref{lda_alpha_eh_acao_nichols}.
	
	To verify that $\lda_\alpha$ is a symmetric partial action, we only need to check condition \eqref{eqn_lda_sim},
	that is,	$\lda_\alpha(h)\lda_\alpha(y)=\lda_\alpha(h_1y)\lda_\alpha(h_2),$ for all $h,y \in \mathbb{H}_{2^n}.$
	By simplicity, we denote $\lda_\alpha$ by $\lda$.
	
	We proceed as in the verification of \eqref{eq} in the proof of Proposition \ref{lda_alpha_eh_acao_nichols}, that is, we fix an element $h \in \mathcal{B}$ and verify condition \eqref{eqn_lda_sim} for all $y \in \mathcal{B}$, where $\mathcal{B}$ is the canonical basis of
	$\mathbb{H}_{2^n}$.
	
	First, for $h \in \{1, g, gx_1, \cdots, gx_{n-1} \}$,  \eqref{eqn_lda_sim} holds trivially for all $y \in \mathcal{B}$.
	
	If $h = x_i$, $i \in \I_{n-1}$, then \eqref{eqn_lda_sim} means $\alpha_i\lda(y) =\lda(x_i y) + \lda(gy)\alpha_i$.
	It is easy to check this equality considering $y$ in cases: $1, g, g^\ell x_j$ and $g^\ell x_{j_1}x_{j_2}...x_{j_s}$, where $j \in \I_{n-1}$, $\ell \in \I_{0,1}$ and $s \in \I_{2, n-1}$, $j_t \in \I_{n-1}$, $t \in \I_{s}$.
	
	Now, for $h=g^j x_{i_1}x_{i_2}$, where $j \in \I_{0, 1}$ and $i_1, i_2 \in \I_{n-1}$ with $i_1 < i_2$, \eqref{eqn_lda_sim} results in
	\begin{align*}
	\lda(g^j x_{i_1}x_{i_2})\lda(y) & = \lda(g^jx_{i_1}x_{i_2} y)  \lda(g^j) - \lda(g^{j+1}x_{i_1}y) \lda(g^jx_{i_2}) \\
	& + \lda(g^{j+1} x_{i_2} y) \lda(g^{j}x_{i_1}) +  \lda(g^j y) \lda( g^jx_{i_1}x_{i_2}).
	\end{align*}

	Since $\lda(gx_s)=\lda(x_s)=\alpha_s$, $\lda(g^j x_{i_1}x_{i_2})=0$ and $\lda(g^j x_{i_1}x_{i_2} y)=0$, for all $s \in \I_{n-1}$ and $y \in \mathcal{B}$, the above equality is reduced to 
	$$\lda(g^{j+1}x_{i_1} y) \alpha_{i_2} = \lda(g^{j+1}x_{i_2} y) \alpha_{i_1}.$$
	This latter equality is easily checked considering $y$ as $1, g$ and $g^\ell x_{j_1}x_{j_2}...x_{j_s}$, $\ell \in \I_{0,1}$, $s \in \I_{n-1}$, $j_t \in \I_{n-1}$, $t \in \I_{s}$, $j_1 < j_2 < \cdots < j_s$.
	
	Finally, consider $h = g^\ell x_{i_1} x_{i_2} \cdots x_{i_s}$, where $\ell \in \I_{0,1}$ and $s \geq 3$, and write $h = g^\ell x_{i_1} x_{i_2} x_{i_3} w$, for some $w \in \mathbb{H}_{2^n}$.
	In this case, \eqref{eqn_lda_sim} means
	\begin{align*}
	\lda(g^\ell  x_{i_1}x_{i_2}x_{i_3}w)\lda(y) & = \lda( g^\ell x_{i_1}x_{i_2}x_{i_3}w_1 y)  \lda(g^\ell w_2) \\
	& -  \lda( g^{\ell +1}x_{i_1}x_{i_3}w_1 y)  \lda(g^\ell x_{i_2}w_2) \\
	& + \lda( g^{\ell +1} x_{i_2}x_{i_3}w_1 y)  \lda( g^{\ell }x_{i_1}w_2) \\
	& + \lda( g^\ell  x_{i_3}w_1 y)  \lda( g^\ell x_{i_1}x_{i_2}w_2) \\
	& +  \lda( g^{\ell +1}x_{i_1}x_{i_2}w_1 y)  \lda( g^\ell x_{i_3} w_2) \\
	& + \lda( g^{\ell }x_{i_1}w_1 y) \lda( g^\ell x_{i_2}x_{i_3} w_2) \\
	& - \lda( g^{\ell } x_{i_2}w_1 y) \lda( g^{\ell }x_{i_1}x_{i_3} w_2) \\
	& + \lda( g^{\ell +1} w_1 y) \lda( g^\ell x_{i_1}x_{i_2}x_{i_3} w_2).
	\end{align*}

	Note that each term in the above equality has the map $\lda$ evaluated in a product involving two skew-primitive elements, which results in $0$, and therefore this equality clearly holds.

	Hence, condition \eqref{eqn_lda_sim} holds for all $h, y \in \mathcal{B}$, that is, $\lda_{\alpha}$ is a symmetric partial action of $\mathbb{H}_{2^n}$ on $\k$.
\end{proof}

To end this subsection, we present all partial actions of the Nichols Hopf algebra $\mathbb{H}_{2^n}$ on $\k$, when $n=2,3$ and $4$. 

\begin{exas}
	Nichols Hopf algebra of order $2$ is the Sweedler's Hopf algebra, and then all partial actions of $\mathbb{H}_4$ on $\k$ are already given in Example \ref{lda_sweedler}.
	
	When $n=3$, all partial actions of $\mathbb{H}_{2^3}$ on $\k$ are given by $\varepsilon$ and $\lda_\alpha$, for any $\alpha=(\alpha_1, \alpha_2) \in \k^2$.
	Namely, $\lda_\alpha(1)=1_\k$, $\lda_\alpha(g)=0$, $\lda_\alpha(x_1x_2)=0$, $\lda_\alpha(gx_1x_2)=0$ and $\lda_\alpha(x_i)=\lda_\alpha(gx_i)=\alpha_i$, for all $i \in \I_{2}$.
	
	For the Nichols Hopf algebra $\mathbb{H}_{2^4}$, we have the global action $\varepsilon$ and the genuine partial actions given by
	$\lda_\beta = 1^* + \sum_{i=1}^{3}\beta_i \, [ (x_i)^* + (gx_i)^*]$, for any $\beta = (\beta_1, \beta_2, \beta_3) \in \k^{3}$.
\end{exas}

\subsection{Partial Coactions}
In this subsection we compute all partial coactions of the Nichols Hopf algebra $\mathbb{H}_{2^n}$ on its base field $\k$. 
The setting here is the same as for the Taft algebra, since both are self-dual Hopf algebras.

\begin{lem}\label{iso_nichols}
	The linear map  $\psi : \mathbb{H}_{2^n}  \longrightarrow (\mathbb{H}_{2^n} )^*$
given by  $\psi(g)=1^* - g^*$ and $\psi(x_i)=x_i^* - (gx_i)^*$, for all $i \in \I_{n-1}$, is an isomorphism of Hopf algebras.
\end{lem}

From the isomorphism $\psi$ in the above lemma, it follows that $$\psi\left(\frac{1+g}{2}\right) = 1^*, \psi\left(\frac{x_i-gx_i}{2}\right) = x_i^* \textrm{ and } \psi\left(\frac{-(x_i+gx_i)}{2}\right) = (gx_i)^*,$$
for all $i \in \I_{n-1}$.
Thus, we are able to prove the following result.
\begin{teo}\label{nichols_coac}
	An element $z \in \mathbb{H}_{2^n}$ is a partial coaction of $ \mathbb{H}_{2^n}$ on $\k$ if and only if $z=1$ (global coaction) or $z=z_\alpha$, where $z_\alpha =\frac{1+g}{2} - \sum_{i=1}^{n-1}\alpha_i \, gx_i$, for any $\alpha= (\alpha_1, \alpha_2, \cdots, \alpha_{n-1}) \in \k^{n-1}$. 
	Moreover, all these partial coactions are symmetric.
\end{teo}

\begin{proof}
	Recall that all partial actions of $\mathbb{H}_{2^n}$ on $\k$ were calculated in Theorem \ref{nichols}: $\varepsilon$ and $\lda_\alpha = 1^* + \sum_{i=1}^{n-1}\alpha_i \, [ (x_i)^* + (gx_i)^*]$, for any $\alpha = (\alpha_i)_{i=1}^{n-1} \in \k^{n-1}$.
	By Proposition \ref{dual_corpo}, all partial coactions of $(\mathbb{H}_{2^n})^*$ on $\k$ are $\rho_\varepsilon$ and $\rho_{\lda_\alpha}$.
	Since $\mathbb{H}_{2^n}$ is a self-dual Hopf algebra, that is, $\psi : \mathbb{H}_{2^n} \longrightarrow (\mathbb{H}_{2^n})^*$ given in Lemma \ref{iso_nichols} is an isomorphism of Hopf algebras, we obtain all partial coactions of $\mathbb{H}_{2^n}$ on $\k$: $\rho_{\psi^{-1}(\varepsilon)}$ and $\rho_{\psi^{-1}(\lda_\alpha)}$.
	Moreover, Theorem \ref{simetria_nichols} and Proposition \ref{dual_corpo} ensure that $\rho_{\psi^{-1}(\varepsilon)}$ and $\rho_{\psi^{-1}(\lda_\alpha)}$ are both symmetric partial coactions.
	
	However, we are interested in an explicit presentation of these partial coactions, \emph{i.e.}, provide the idempotent elements $\psi^{-1}(\varepsilon), \psi^{-1}(\lda_\alpha) \in \mathbb{H}_{2^n}$.
	We deal only with $\psi^{-1}(\lda_\alpha)$, since clearly $\psi^{-1}(\varepsilon) = 1$. 
	
	Since $\lda_\alpha = 1^* + \sum_{i=1}^{n-1}\alpha_i \, [ (x_i)^* + (gx_i)^*]$ and we know $\psi^{-1}(1^*), \psi^{-1}((x_i)^*)$ and $\psi^{-1}((gx_i)^*)$, we are able to compute $\psi^{-1}(\lda_\alpha)$:
	\begin{align*}
	\psi^{-1}(\lda_\alpha) & = \psi^{-1} \left( 1^* + \sum_{i=1}^{n-1}\alpha_i \, \left( (x_i)^* + (gx_i)^* \right) \right) \\
	& = \psi^{-1}(1^*) +  \sum_{i=1}^{n-1}\alpha_i \, \left( \psi^{-1}((x_i)^*) + \psi^{-1}((gx_i)^*)\right) \\
	& = \frac{1+g}{2} + \sum_{i=1}^{n-1}\alpha_i \left( \frac{x_i-gx_i}{2} - \frac{(x_i+gx_i)}{2} \right) \\
	& = \frac{1+g}{2} - \sum_{i=1}^{n-1}\alpha_i \, gx_i.
	\end{align*}
Denoting $\psi^{-1}(\lda_\alpha)$ by $z_\alpha$, we finish the proof.
\end{proof}

To conclude, we show all partial coactions of the Nichols Hopf algebra $\mathbb{H}_{2^n}$ on $\k$, when $n=2,3$ and $4$. 

\begin{exas}
	When $n=2$, all partial coactions of the Nichols Hopf algebra $\mathbb{H}_{2^2}$ on $\k$ are $z=1$ (global) and $z_\alpha = \frac{1+g}{2}-\alpha gx$, $\alpha \in \k$.
	Recall that the Nichols Hopf algebra of order $2$ is the Sweedler's $4$-dimensional Hopf algebra, and then such partial coactions were given in Example \ref{z_sweedler}.
	
	When $n=3$, all partial coactions of $\mathbb{H}_{2^3}$ on $\k$ are $z=1$ (global) and $z_\beta = \frac{1+g}{2}-\beta_1 gx_1-\beta_2 gx_2$, for any $\beta=(\beta_1, \beta_2) \in \k^{2}$.
	
	When $n=4$, all partial coactions of $\mathbb{H}_{2^4}$ on $\k$ are $z=1$ (global) and $z_\gamma = \frac{1+g}{2}-\gamma_1 gx_1-\gamma_2 gx_2-\gamma_3 gx_3$, for any $\gamma=(\gamma_1, \gamma_2, \gamma_3) \in \k^{3}$.
\end{exas}	

\subsection*{Acknowledgments} 
We thank Antonio Paques and João M. J. Giraldi for interesting conversations and advices at different
moments of our research.

\bibliographystyle{abbrv}

\end{document}